\documentclass[a4paper, reqno]{amsart}

\usepackage[utf8]{inputenc}
\usepackage{eucal, stmaryrd}
\usepackage{mathtools}
\usepackage{graphicx}
\usepackage[pdftex,colorlinks,citecolor=black,linkcolor=black,
            urlcolor=black,bookmarks=false]{hyperref}
\usepackage{hyphenat}
\usepackage[table]{xcolor}
\usepackage{blkarray}
\usepackage{longtable}

\setkeys{Gin}{draft=false}

\usepackage{caption}
\newtheorem{lemma}{Lemma}[section]
\newtheorem{theorem}[lemma]{Theorem}
\newtheorem{conjecture}[lemma]{Conjecture}

\theoremstyle{definition}
\newtheorem{definition}[lemma]{Definition}

\newcommand{\Z}{\mathbb{Z}}

\newcommand{\R}{\mathbb{R}}
\newcommand{\C}{\mathbb{C}}
\newcommand{\Rplus}{\R_+}

\newcommand{\Mcal}{\mathcal{M}}

\newcommand{\Scal}{\mathcal{S}}

\newcommand{\Ucal}{\mathcal{U}}
\newcommand{\Xcal}{\mathcal{X}}

\newcommand{\indf}[1]{\mathbf{1}_{#1}}     
\newcommand{\tp}{\mathsf{T}}               
\newcommand{\sorto}{\mathrm{SO}}           
\newcommand{\suni}{\mathrm{SU}}            
\newcommand{\one}{\mathbf{1}}              
\DeclareMathOperator{\trace}{tr}           
\newcommand{\floor}[1]{\lfloor #1\rfloor}  
\newcommand{\polar}{\bullet}               
\DeclareMathOperator{\angdst}{\angle}      
\DeclareMathOperator{\cayley}{Cayley}      

\newcommand{\smallpmatrix}[1]{\left(\begin{smallmatrix}#1\end{smallmatrix}\right)}

\newcommand{\zint}[4]{#1#3\mathinner{\ldotp\ldotp}#4#2}

\newenvironment{optprob}
{
  \arraycolsep=0pt
  \begin{array}{r@{\ }l@{\quad}l}
}%
{
  \end{array}
}

\newcommand{\onerow}[1]{\multicolumn{2}{l}{#1}}

\newlength\claimlen

{
  \end{minipage}%
  \end{equation}%
  \ignorespacesafterend%
}

\newcommand{\defi}[1]{\textit{#1}}

 \makeatletter
 \newcommand*{\centerfloat}{%
   \parindent\z@%
   \leftskip\z@\@plus 1fil\@minus\textwidth%
   \rightskip\leftskip%
   \parfillskip\z@skip}
 \makeatother

\title{Bounding the density of spherical polygon packings}

\author{Fernando Mário de Oliveira Filho}
\address{F.M. de Oliveira Filho, Delft Institute of Applied Mathematics,
  Delft University of Technology, Mekelweg~4, 2628~CD Delft, The
  Netherlands.}
\email{F.M.deOliveiraFilho@tudelft.nl}

\author{Andreas Spomer}
\address{A. Spomer, Department Mathematik/Informatik, Abteilung
  Mathematik, Universität zu Köln, Weyertal 86--90, 50931 Köln,
  Germany.}
\email{andreas.spomer1996@gmail.com}

\author{Frank Vallentin}
\address{F. Vallentin, Department Mathematik/Informatik, Abteilung
  Mathematik, Universität zu Köln, Weyertal 86--90, 50931 Köln,
  Germany.}
\email{frank.vallentin@uni-koeln.de}


\date{April 23, 2026}

\begin{document}

\begin{abstract}
  We determine putative optimal packings of regular spherical polygons via
  optimization on smooth manifolds.  For several cases, we establish maximality
  by extending the Lovász theta number to Cayley graphs on the special
  orthogonal group~$\sorto(3)$.  To this end, we introduce an algebraic
  criterion characterizing when congruent regular spherical polygons have
  disjoint interiors, leading to a unified formulation of the packing
  constraints. Using harmonic analysis on~$\sorto(3)$, we reduce the theta
  number to a trigonometric sum-of-squares problem, which can be solved via
  semidefinite programming.
\end{abstract}

\maketitle
\markboth{F.M. de Oliveira Filho, A. Spomer, and F. Vallentin}{Bounding the
density of spherical polygon packings}

\setcounter{tocdepth}{1}
\tableofcontents


\section{Introduction}%
\label{sec:intro}

Let~$X$ be a shape---a spherical cap, a spherical polygon, etc.---on the
$2$-dimen\-sional unit sphere~$S^2$ in~$\R^3$.  The packing problem for~$X$ asks:
how many nonoverlapping, congruent copies of~$X$ fit on the sphere?
Figure~\ref{fig:packings} shows packings of spherical caps and regular triangles
on~$S^2$.

\begin{figure}[t]
  \centerfloat
  \includegraphics[width=6cm]{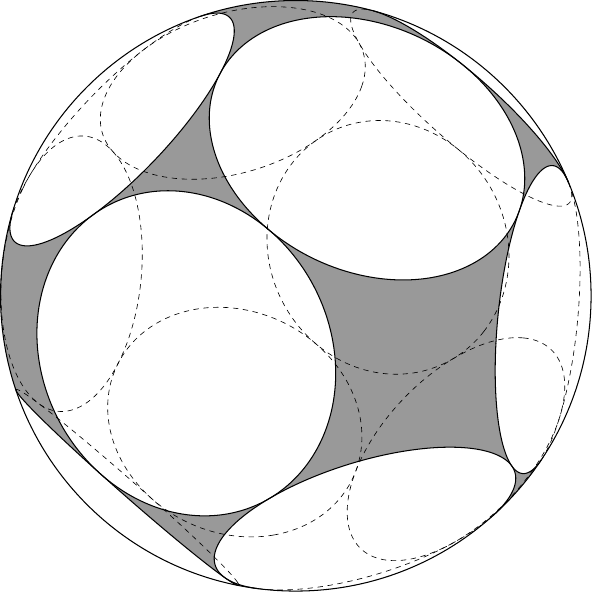}
  \qquad
  \includegraphics[width=6cm]{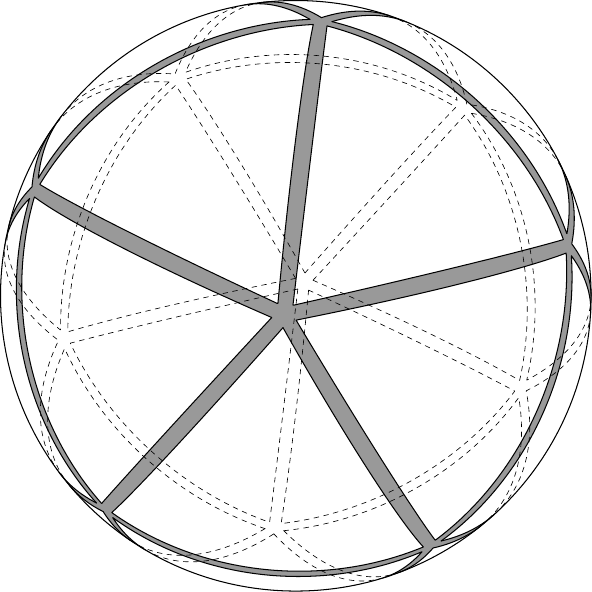}

  \caption{On the left, the spherical cap packing corresponding to the maximal
  fcc kissing configuration in~$\R^3$.  On the right, a packing of~20 spherical
  triangles corresponding to the bases of regular tetrahedra with a vertex
  at~$0$; these triangles are obtained by shrinking the~20 faces of a regular
  icosahedron.}%
  \label{fig:packings}
\end{figure}

When the shape given is a spherical cap of angular radius~$\phi$, this is the
classical problem of computing the maximum cardinality of a spherical code with
minimum angular distance~$2\phi$.  For instance, the maximum number of
nonoverlapping spherical caps with radius~$\pi/6$ that fit on~$S^2$ is the
kissing number~$\tau_3$ of~$\R^3$, that is, the maximum number of nonoverlapping
unit spheres that can simultaneously touch a central unit sphere.
Determining~$\tau_3$ is a problem that goes back to a 17th-century dispute
between Newton and Gregory~\cite{Casselman2004}.  Only in~1953 did Schütte and
van der Waerden~\cite{SchutteW1953} prove that~$\tau_3 = 12$.

A much older instance of the problem goes back to Aristotle (384 BC--322 BC),
who claimed (\textit{De Caelo}, Book~III, Chapter~VIII) that regular tetrahedra
could fill a container without leaving gaps. This led to the question whether
the space around a point can be filled with regular tetrahedra sharing the
point as a vertex.  Through time, many important scholars commented on this
problem, leading to a fascinating discussion (for a summary, see Lagarias and
Zong~\cite{LagariasZ2012} and Oliveira and Vallentin~\cite{OliveiraV2018}).  A
proof of the negative answer came only in the~15th century through the now-lost
work of Johannes Müller von Königsberg (1436--1476), known as Regiomontanus; a
proof by Francesco Maurolico (1494--1575) was recently
rediscovered~\cite{Addabbo2015}.

So it is not possible to completely fill the space around a point with
nonoverlapping regular tetrahedra, but how many nonoverlapping regular
tetrahedra can share a vertex in common?  A regular tetrahedron with a vertex
at~$0$ defines an equilateral triangle~$T$ on~$S^2$.  Thus the question can be
equivalently reformulated as: how many nonoverlapping congruent copies of~$T$
fit on~$S^2$? This maximum is also known as the \defi{kissing number} of the
tetrahedron.
\bigbreak

\centerline{$*\qquad*\qquad*$}
\bigbreak

A \defi{packing} of a set~$X \subseteq S^2$ is a collection of interior-disjoint
congruent copies of~$X$.  A packing is \defi{maximal} if no larger packing
exists.  A packing of a spherical cap~$C$ is \defi{optimal} if no packing with
at least as many caps of radius greater than that of~$C$ exists.  Optimality for
packings of a given regular spherical polygon is similarly defined, with the
circumradius of the polygon replacing the radius of the spherical cap.

The packing of spherical caps in Figure~\ref{fig:packings} is maximal, as shown
by Schütte and van der Waerden~\cite{SchutteW1953}, but not optimal.  The latter
fact is unclear from the picture, though the caps can be rearranged so as not to
touch each other.  The packing of spherical triangles from the same figure is
clearly not optimal, as the triangles can be enlarged, and it is not known
whether it is maximal.

Constructing good packings of spherical caps is a difficult problem, as is
proving maximality of such packings.  Much harder even is to prove optimality:
Schütte and van der Waerden~\cite{SchutteW1951} constructed spherical cap
packings with~5--11, 13--16, 24, and~32 caps and managed to prove optimality of
the configurations with~5--9 caps; the only other optimal configurations known
have~10--14 and~24 caps (see the table of spherical codes by
Cohn~\cite{Cohn2024} for further references).

In this paper, we consider the problem of finding optimal packings of regular
spherical polygons.  Here, even less is known than for spherical caps.  First,
it is harder to come up with good constructions.  Second, proving maximality is
substantially more difficult, and proofs of optimality are only known in the
trivial situation when a polygon tiles the sphere.  We approach the problem from
two sides, by presenting a method to find good packings---that is, lower
bounds---and by developing upper bounds for the cardinality of a packing.

\subsubsection*{Characterizing intersection}

Our methods to find lower and upper bounds both rely on optimization and as such
are computational in nature.  They presuppose a way to translate the packing
constraints, namely the condition that two congruent copies of a regular
spherical polygon are interior-disjoint, into a format that is amenable to
computation.

This is achieved by considering \textit{admissible} polygons
(Definition~\ref{def:admissible}), which are those for which a clean algebraic
condition for interior-disjointness can be given.
Theorems~\ref{thm:admissible-N4} and~\ref{thm:admissible-andreas} show that most
regular spherical polygons of interest are admissible.  The proofs of these
theorems require extensive geometrical work; they are based on the idea of
sentinels introduced by Mascarenhas and Birgin~\cite{MascarenhasB2010} for the
simpler situation of intersecting polygons on the plane.  Our method to find
lower bounds is based on admissibility and, though the concept is not needed for
the abstract definition of our optimization upper bound, it is crucial for
transforming the initial optimization problem into one that can be solved by a
computer.

\subsubsection*{Lower bounds: constructions}

In~\S\ref{sec:manifold-opt}, we model the problem of finding an optimal packing
with a given number of regular spherical polygons as a manifold optimization
problem. Numerical optimization methods from smooth Riemannian optimization can
then be used to find approximate local optima.  Essential for this modeling is
the characterization of intersection of regular spherical polygons encoded in
the definition of admissibility (Definition~\ref{def:admissible}), as it allows
the packing constraints to be modeled in terms of well-behaved functions.

An indication that our method produces near-optimal packings is given by
benchmark results on spherical caps, as in this case it is able to closely
approximate known optimal packings (see Table~\ref{tab:caps}
of~\S\ref{sec:results}).

\subsubsection*{Upper bounds: the Lovász theta number}

Maximality and optimality of spherical cap packings in~$S^2$ are sometimes
proved directly by \textit{ad hoc} geometrical arguments.  Such proofs are often
difficult, like the proof by Schütte and van der Waerden~\cite{SchutteW1953}
that the kissing number of~$\R^3$ is~12, that is, that the spherical cap packing
of Figure~\ref{fig:packings} is maximal.  Direct proofs of optimality, like
those of Schütte and van der Waerden~\cite{SchutteW1951} for packings with~5--9
caps, are rare and often involved; see for instance the proofs of Musin and
Tarasov for~13 and~14 points~\cite{MusinT2012, MusinT2015}.

Upper bounds based on optimization provide the most successful framework for
proving maximality/optimality of spherical cap packings.  These bounds are now
understood to derive from the Lovász theta number, a parameter introduced by
Lovász~\cite{Lovasz1979} that gives an upper bound for the independence number
of a finite graph via semidefinite programming.

The prime example here is the linear programming bound of Delsarte, Goethals,
and Seidel~\cite{DelsarteGS1977} for the cardinality of spherical cap packings,
whose connection to the theta number was established by Bachoc, Nebe, Oliveira,
and Vallentin~\cite{BachocNOV2009}.  The linear programming bound and its
offshoots like the 3-point bound of Bachoc and Vallentin~\cite{BachocV2008} can
be used to prove the maximality of several packings; such proofs amount to
finding an upper bound which, rounded down, coincides with the size of the
corresponding configuration.  In contrast, proofs of optimality require finding
a sharp upper bound, a rare occurrence.  For the most recent application of
optimization upper bounds to prove optimality of configurations and for further
references, see Cohn, de Laat, and Leijenhorst~\cite{CohnLL2024}.

A bit more general than the Delsarte-Goethals-Seidel bound is the linear
programming bound of Kabatiansky and Levenshtein~\cite{KabatianskyL1978} for
packings of metric balls in compact, connected, rank-one symmetric spaces.  The
Euclidean analogue of these bounds is the linear programming bound of Cohn and
Elkies~\cite{CohnE2003} for the sphere-packing density in~$\R^n$.  Though in
general computing the theta number of a graph amounts to solving a semidefinite
programming problem, here the initial semidefinite program collapses to a
linear program because of symmetry, hence the name ``linear programming bound''
that is commonly used.

When it comes to packings of more general shapes, only a few results appear in
the literature.  Dostert, Guzmán, Oliveira, and Vallentin~\cite{DostertGOV2017}
consider translative packings of polytopes in Euclidean space.  In this case,
the semidefinite program again collapses to a linear program.  This is no longer
the case for packings of balls of different sizes, considered by de Laat,
Oliveira, and Vallentin~\cite{LaatOV2014}, and for packings of congruent copies
of regular pentagons on the Euclidean plane, treated by Oliveira and
Vallentin~\cite{OliveiraV2018}.  In this latter setting, the semidefinite
program involves a kernel-valued function of positive type as the optimization
variable.

The problem of packing regular polygons on~$S^2$ is similar to that of packing
regular pentagons on~$\R^2$.  Our bound, presented in~\S\ref{sec:theta}, is an
extension of the Lovász theta number to an infinite Cayley graph on the special
orthogonal group~$\sorto(3)$.  In general, it does not reduce to linear
programming, and the optimization variables are positive-semidefinite kernels
over the special orthogonal group~$\sorto(3)$.  For packings of regular
spherical polygons, we use harmonic analysis, and therefore the representation
theory of~$\sorto(3)$, to parameterize such kernels.  Although the
representation theory of~$\sorto(3)$ has been worked out in great detail before,
we are not aware of any applications in the context of semidefinite programming.

The concept of admissibility (Definition~\ref{def:admissible}) is again
essential for the computation of the bound.  It allows us to transform the
initial optimization problem into a trigonometric sum-of-squares problem, which
can then be reduced to a semidefinite programming problem.  The resulting
problem is at first too large to be solved by computer.  We then have to exploit
additional symmetries, both obvious and not, stemming from the geometry of the
packing problem, to reduce the size of the semidefinite program.  This leads
in~\S\ref{sec:sos} to a framework for dealing with invariant trigonometric sums
of squares, based on the theory for classical polynomials by Gatermann and
Parrilo~\cite{GatermannP2004}. In~\S\ref{sec:sos-compare} we discuss the
important differences between the trigonometrical and classical frameworks for
invariant sums of squares.

\subsubsection*{Results}

Our results are discussed in~\S\ref{sec:results}. Some of the more interesting
packings that we compute via manifold optimization are depicted in
Figures~\ref{fig:maximal-packings} and~\ref{fig:large-packings}.  We manage to
find several maximal packings with the manifold optimization approach, using the
Lovász theta number upper bound to prove maximality. Proofs of optimality or
near-optimality depend on sharp or near-sharp bounds.  These are rare for
spherical cap packings~\cite{CohnLL2024} and, Table~\ref{tab:bounds} suggests,
even rarer for packings of regular spherical polygons, the only known cases
happening when the polygon tiles the sphere.  A future line of research would be
to develop tractable higher-order bounds for packings of regular spherical
polygons, like the $k$-point bound~\cite{LaatMOV2022} or the Lasserre
hierarchy~\cite{LaatV2015}, with the goal of proving maximality and
near-optimality of more configurations.


\begin{figure}[t]
  \setbox0=\hbox{\hskip3.5cm\quad\hskip3.5cm\qquad\hskip3.5cm\quad\hskip3.5cm}
  \centerfloat
  \includegraphics[width=3.5cm]{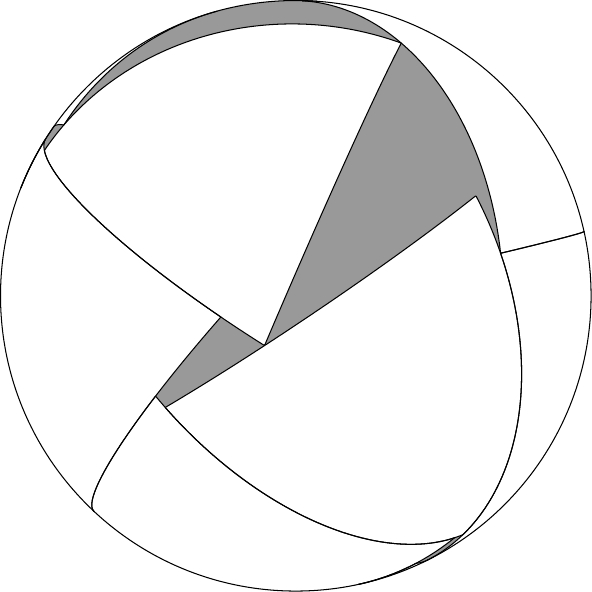}
  \quad
  \includegraphics[width=3.5cm]{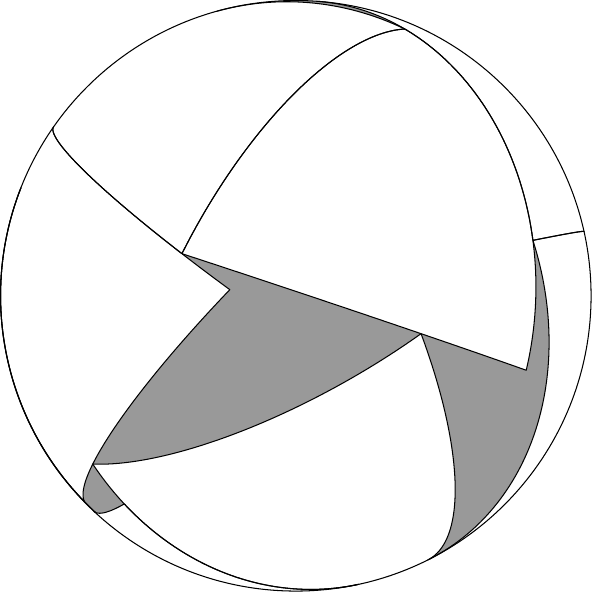}
  \qquad
  \includegraphics[width=3.5cm]{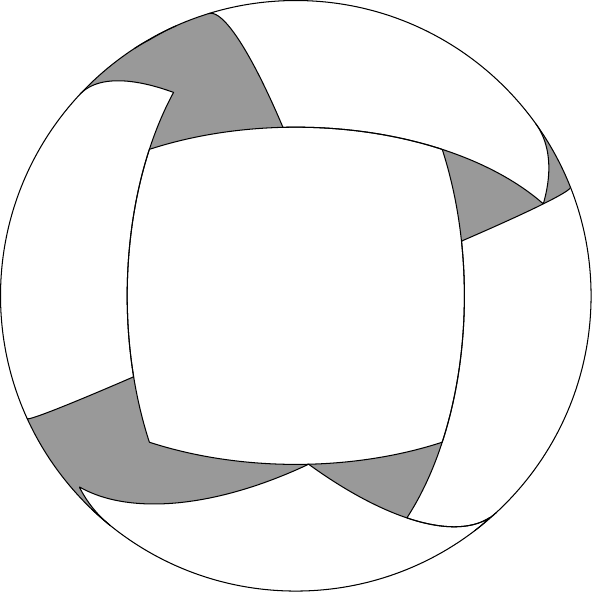}
  \quad
  \includegraphics[width=3.5cm]{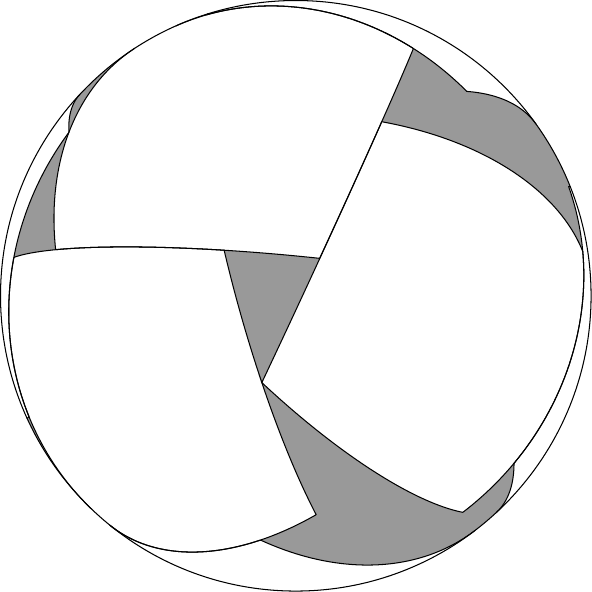}

  \centerline{\hbox to\wd0{\hskip3.3cm(a)\hfill (b)\hskip3.3cm}}
  \vskip5mm

  \includegraphics[width=3.5cm]{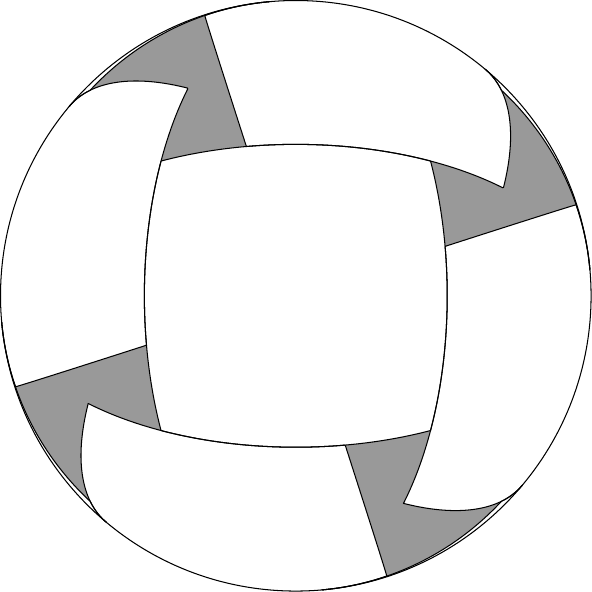}
  \quad
  \includegraphics[width=3.5cm]{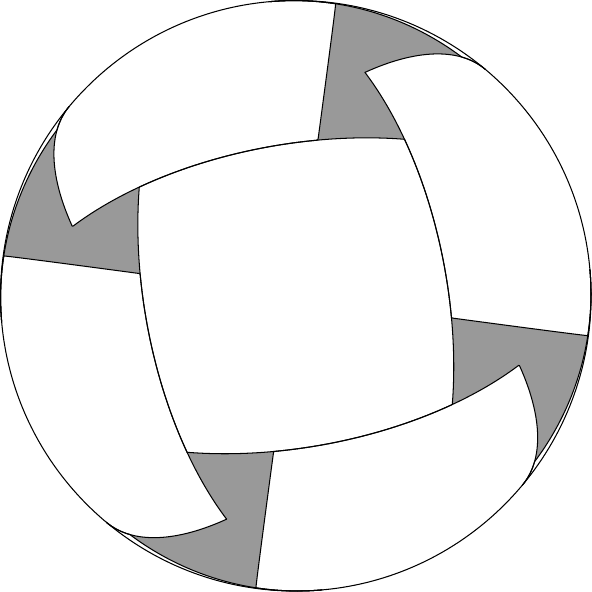}
  \qquad
  \includegraphics[width=3.5cm]{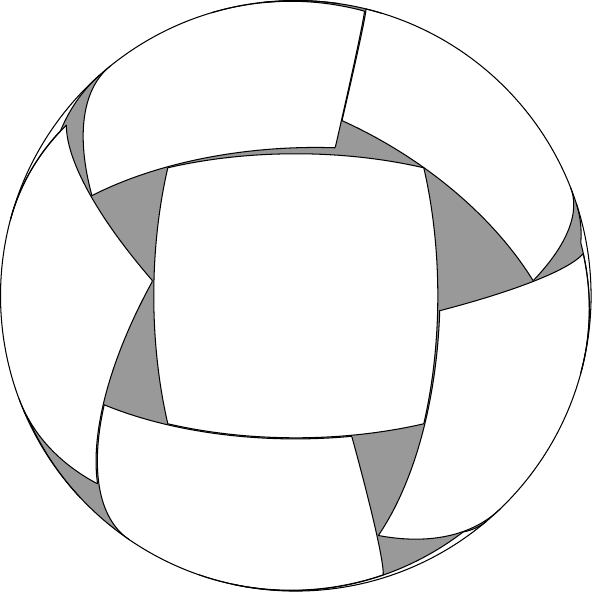}
  \quad
  \includegraphics[width=3.5cm]{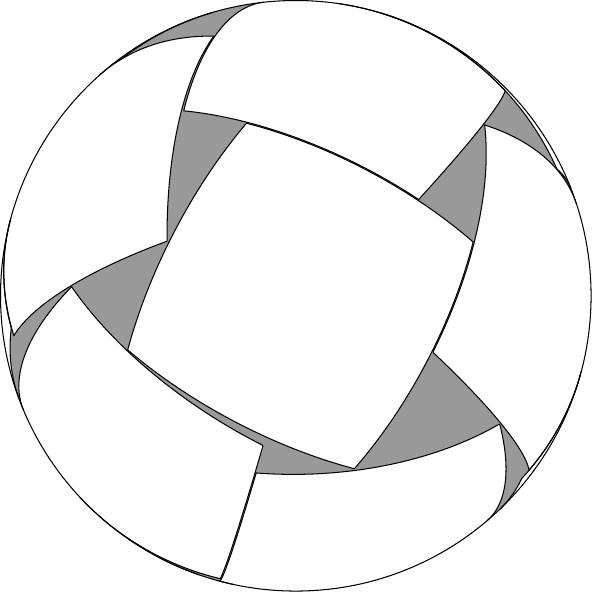}

  \centerline{\hbox to\wd0{\hskip3.3cm (c)\hfill (d)\hskip3.3cm}}
  \vskip5mm

  \includegraphics[width=3.5cm]{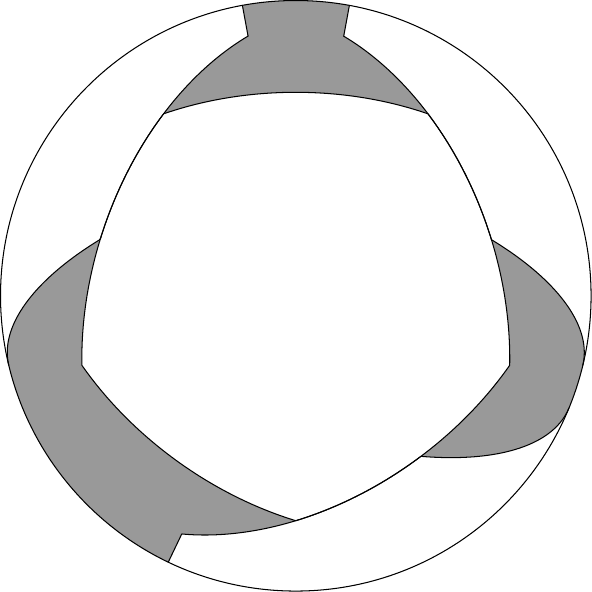}
  \quad
  \includegraphics[width=3.5cm]{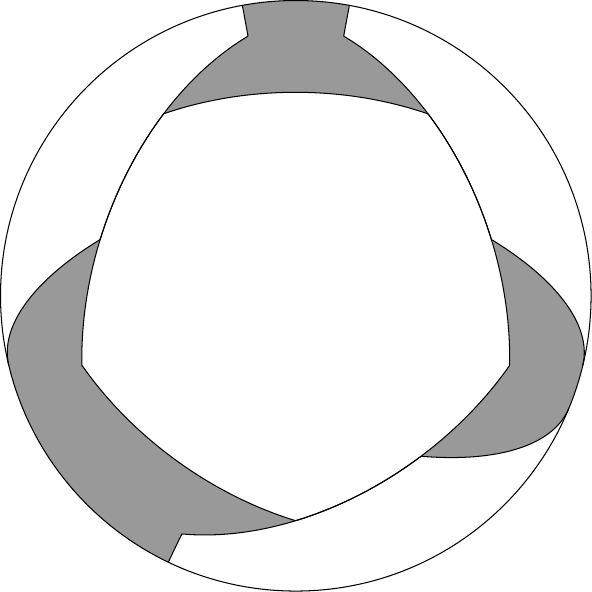}
  \qquad
  \includegraphics[width=3.5cm]{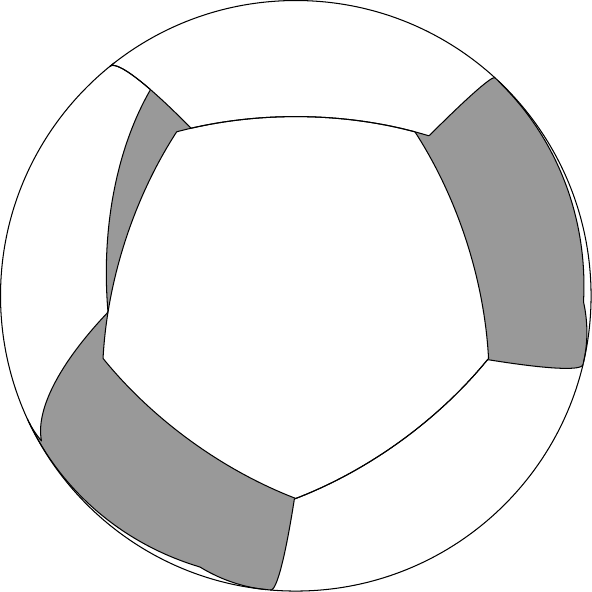}
  \quad
  \includegraphics[width=3.5cm]{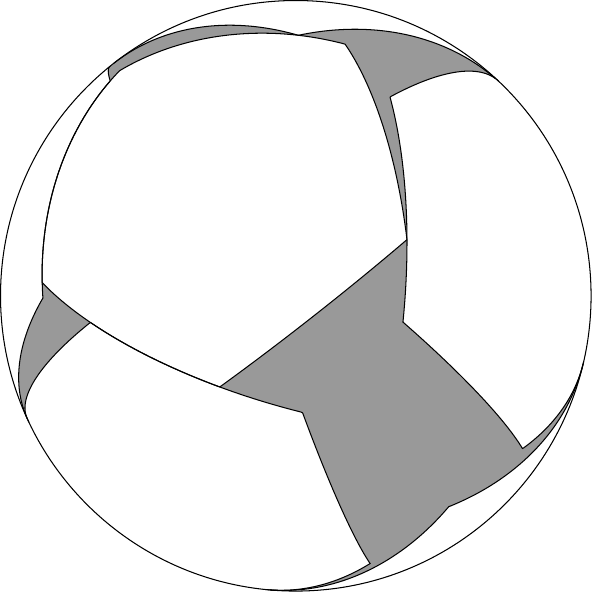}

  \centerline{\hbox to\wd0{\hskip3.3cm (e)\hfill (f)\hskip3.3cm}}

  \caption{Some of the configurations from Table~\ref{tab:bounds} proven to be
  maximal by the theta prime bound of~\S\ref{sec:theta}.  Both the front and
  back sides of each configuration are shown in orthographic projection.  (a)~10
  triangles, (b)~8 squares, (c)~10 squares, (d)~12 squares, (e)~5 pentagons, and
  (f)~7 pentagons.}%
  \label{fig:maximal-packings}
\end{figure}

\begin{figure}[t]
  \setbox0=\hbox{\hskip3.5cm\quad\hskip3.5cm\qquad\hskip3.5cm\quad\hskip3.5cm}
  \centerfloat
  \includegraphics[width=3.5cm]{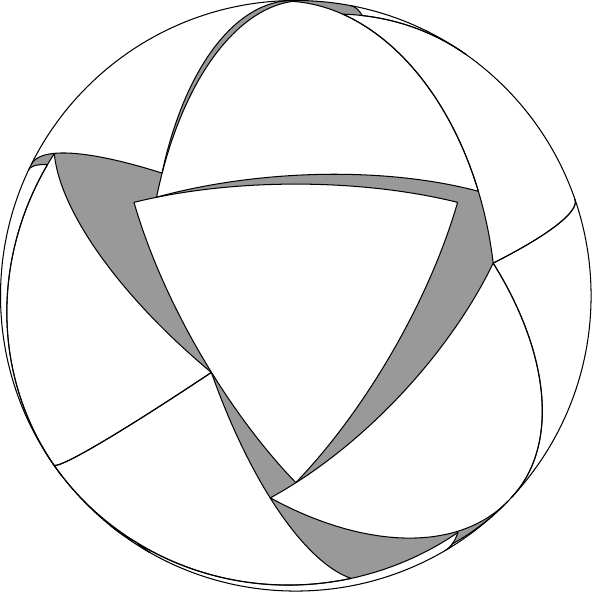}
  \quad
  \includegraphics[width=3.5cm]{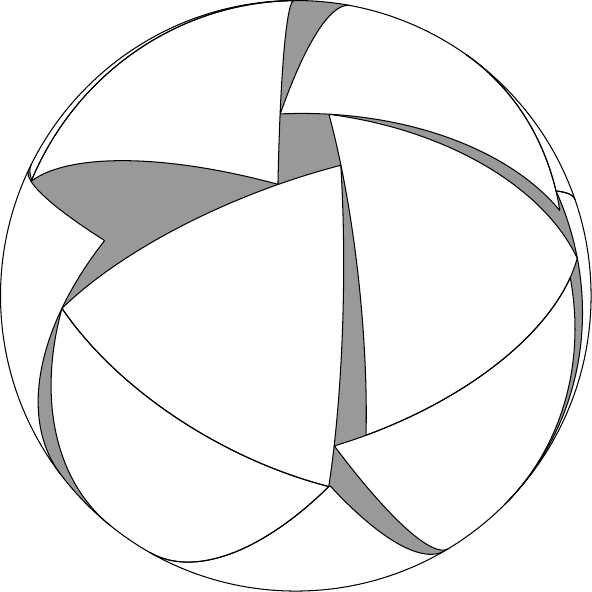}
  \qquad
  \includegraphics[width=3.5cm]{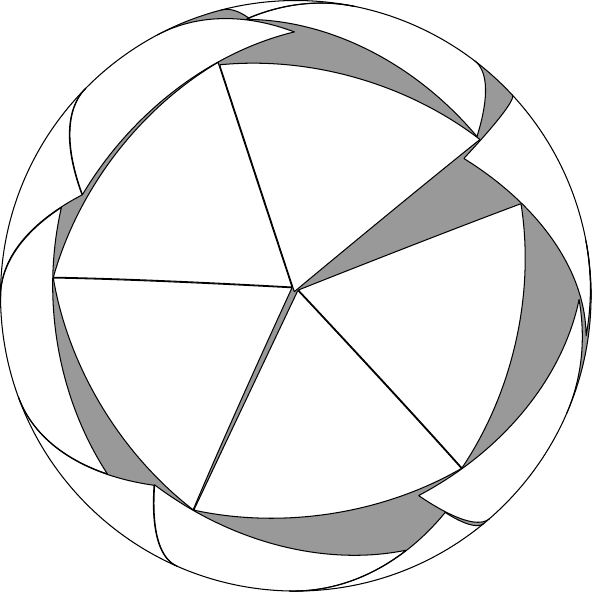}
  \quad
  \includegraphics[width=3.5cm]{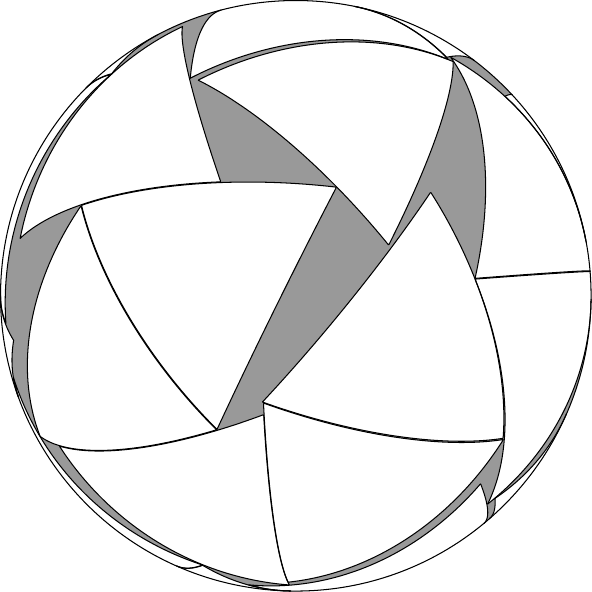}

  \centerline{\hbox to\wd0{\hskip3.3cm(a)\hfill (b)\hskip3.3cm}}
  \vskip5mm

  \includegraphics[width=3.5cm]{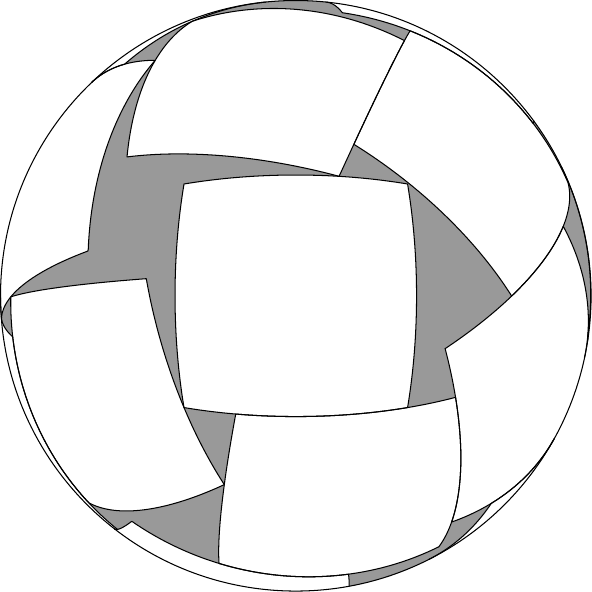}
  \quad
  \includegraphics[width=3.5cm]{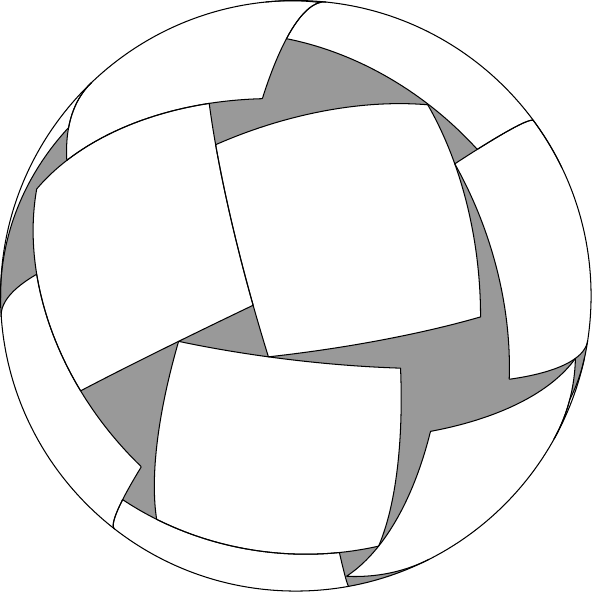}
  \qquad
  \includegraphics[width=3.5cm]{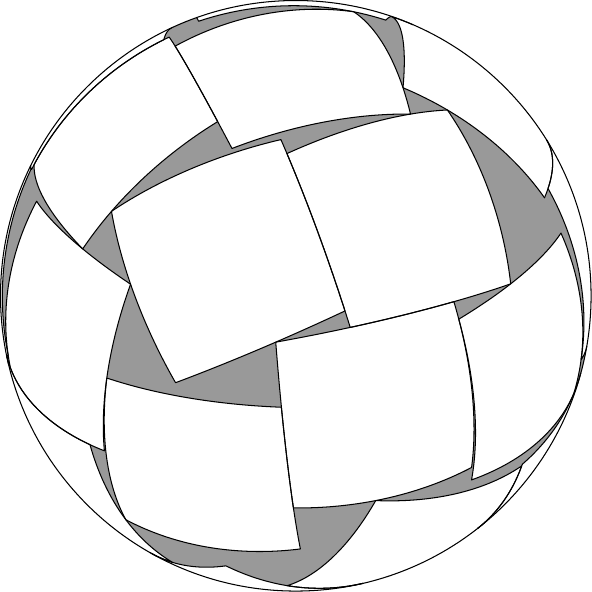}
  \quad
  \includegraphics[width=3.5cm]{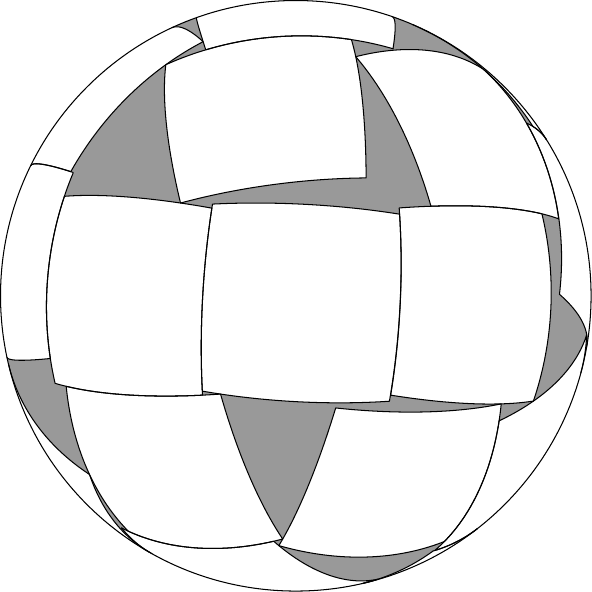}

  \centerline{\hbox to\wd0{\hskip3.3cm (c)\hfill (d)\hskip3.3cm}}
  \vskip5mm

  \includegraphics[width=3.5cm]{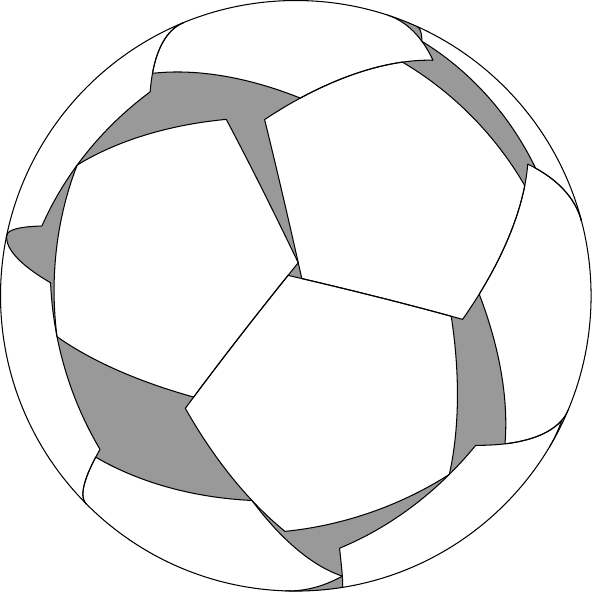}
  \quad
  \includegraphics[width=3.5cm]{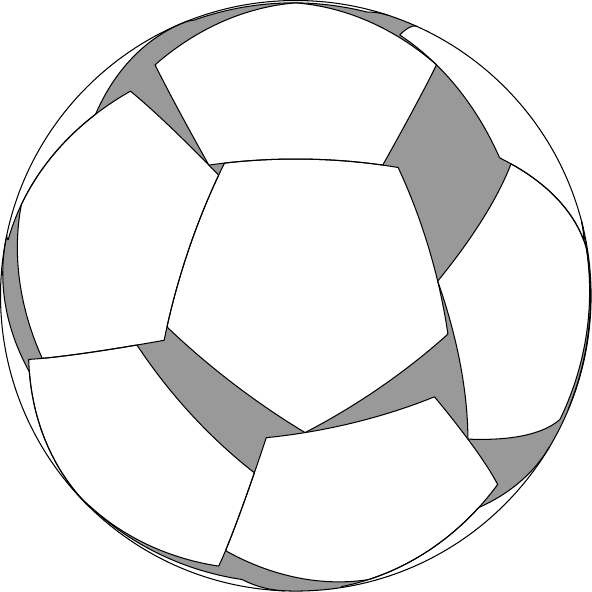}
  \qquad
  \includegraphics[width=3.5cm]{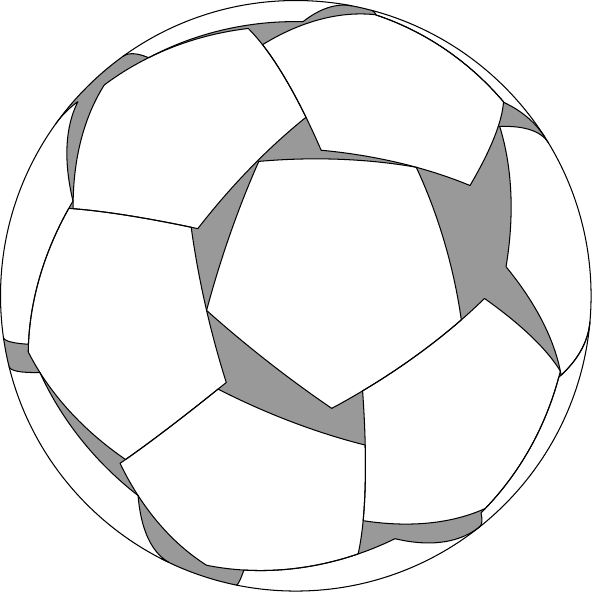}
  \quad
  \includegraphics[width=3.5cm]{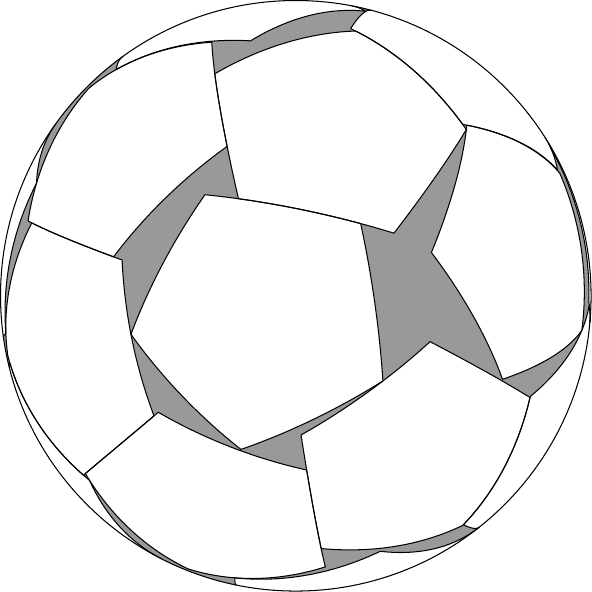}

  \centerline{\hbox to\wd0{\hskip3.3cm (e)\hfill (f)\hskip3.3cm}}

  \caption{Some of the larger configurations from Table~\ref{tab:bounds}.  Both
  the front and back sides of each configuration are shown in orthographic
  projection.  (a)~16 triangles, (b)~24 triangles, (c)~16 squares, (d)~24
  squares, (e)~16 pentagons, and (f)~20 pentagons.}%
  \label{fig:large-packings}
\end{figure}



\section{Preliminaries}
\label{sec:notation}

The interior of a set~$X$ is denoted by~$X^\circ$.  The Euclidean inner product
between~$x$, $y \in \R^n$ is~$x\cdot y = x_1 y_1 + \cdots + x_n y_n$. If~$A$, $B
\in \C^{n \times n}$ are matrices, then~$\langle A, B\rangle = \trace A^* B =
\sum_{i,j=1}^n \overline{A_{ij}} B_{ij}$ is their trace inner product.

The $(n-1)$-dimensional unit sphere is~$S^{n-1} = \{\, x \in \R^n :
\|x\|=1\,\}$.  The angular distance between points~$x$, $y \in S^{n-1}$
is~$\angdst(x, y) = \arccos x\cdot y$.  Given~$x \in S^{n-1}$ and~$\phi > 0$,
the \defi{spherical cap} with \defi{center~$x$} and \defi{angular radius~$\phi$}
is the set of all points in~$S^{n-1}$ at angular distance at most~$\phi$
from~$x$.

The special orthogonal group of~$\R^n$ is
\[
  \sorto(n) = \{\, A \in \R^{n \times n} : A^\tp A = I\text{ and }\det A =
  1\,\}.
\]
Given an angle~$\theta$, consider the rotation matrices
\begin{equation}%
  \label{eq:rotation-matrices}
  R_X(\theta) = \begin{pmatrix}
    1 & 0           & 0\\
    0 & \cos\theta  & \sin\theta\\
    0 & -\sin\theta & \cos\theta
  \end{pmatrix}
  \qquad\text{and}\qquad
  R_Z(\theta) = \begin{pmatrix}
    \cos\theta  & \sin\theta & 0\\
    -\sin\theta & \cos\theta & 0\\
    0           & 0          & 1
  \end{pmatrix}
\end{equation}
that fix~$e_1 = (1, 0, 0)$ and~$e_3 = (0, 0, 1)$, respectively.  These matrices
are elements of~$\sorto(3)$.

Given angles~$\alpha \in [0, 2\pi)$, $\beta \in [0, \pi]$, and~$\gamma \in [0,
2\pi)$, let
\begin{equation}%
  \label{eq:euler-rotation}
  R(\alpha, \beta, \gamma) = R_Z(\gamma) R_X(\beta) R_Z(\alpha).
\end{equation}
Note that~$R(\alpha, \beta, \gamma)$ is an element of~$\sorto(3)$.  Moreover,
every element of~$\sorto(3)$ can be written in this form for some
triple~$(\alpha, \beta, \gamma)$ of angles in the above intervals; these angles
are called \defi{Euler angles}, and the matrix~$R(\alpha, \beta, \gamma)$ is an
\defi{Euler rotation}.  The restriction to these intervals makes the angles
associated with a given rotation unique, except when~$\beta = 0$ or~$\beta =
\pi$, in which cases~$\alpha$ and~$\gamma$ are not uniquely determined.

Let~$N \geq 3$ be an integer and let~$\varrho \in (0, \pi/2)$.  The
\defi{standard regular $N$-gon} with \defi{angular radius}~$\varrho$ is the cone
generated by the vectors
\[
  u_i = R_Z((2i - 1)\pi / N) (0, \sin\varrho, \cos\varrho)\qquad\text{for~$i =
  0$, \dots,~$N - 1$}.
\]
A \defi{regular $N$-gon} with \defi{angular radius}~$\varrho$ is any rotation of
the standard regular $N$-gon with angular radius~$\varrho$.  A \defi{regular
spherical polygon} is just a regular $N$-gon for some~$N \geq 3$.

Let~$K$ be a regular spherical polygon.  The \defi{vertices} of~$K$ are the
intersections of the extreme rays of~$K$ with the sphere; the \defi{barycenter}
of~$K$ is the sum of its vertices, normalized to have norm~$1$.  So the vertices
of the standard regular $N$-gon are the points~$u_i$ and its barycenter is~$e_3
= (0, 0, 1)$.  The \defi{inner angle} of~$K$ is the dihedral angle between two
consecutive facets of~$K$.  The \defi{side length} of~$K$ is the angular
distance between two of its consecutive vertices.

Let~$\varrho$ be the angular radius of~$K$.  Note that~$\varrho$ is the angular
distance between the vertices of~$K$ and its barycenter.  The
\defi{circumcircle} of~$K$ is the set of all points on~$S^2$ at angular
distance~$\varrho$ from its barycenter.  The \defi{inradius} of~$K$ is the
angular distance between the barycenter of~$K$ and the midpoint of any of its
sides; such a midpoint is for instance~$(u + v) / \|u+v\|$, where~$u$, $v$ are
any two consecutive vertices of~$K$.  Let~$\tau$ be the inradius of~$K$.  The
\defi{incircle} of~$K$ is the set of all points on~$S^2$ at angular
distance~$\tau$ from its barycenter.


\subsection{A packing problem and the volume bound}%
\label{sec:packing-volume}

Here is a formal description of the problem we consider.  We say that sets~$X$,
$Y \subseteq \R^n$ are \defi{congruent} if~$X = S Y$ for some~$S \in \sorto(n)$.
A \defi{convex cone} is a nonempty set~$K \subseteq \R^n$ closed under conic
combinations: if~$x$, $y \in K$ and~$\alpha$, $\beta \geq 0$, then~$\alpha x +
\beta y \in K$.  A \defi{proper cone} is a convex cone that is closed, has
nonempty interior, and does not contain a line.  Not containing a line means
that a proper cone is \defi{pointed}, that is, that it has a vertex at~$0$. This
implies that~$0$ is not in the interior of a proper cone.

Let~$K \subseteq \R^n$ be a proper cone.  A \defi{packing} of~$K$ is a
collection of interior-disjoint congruent copies of~$K$.  The \defi{packing
number} of~$K$ is
\[
  \tau(K) = \text{maximum cardinality of a packing of~$K$}.
\]
The cone~$K$ defines the set~$X = K \cap S^{n-1}$.  The packing number of~$K$ is
the maximum number of nonoverlapping congruent copies of~$X$ that fit
on~$S^{n-1}$.

If~$\omega$ is the surface measure on~$S^{n-1}$, then an upper bound
for~$\tau(K)$ is
\[
  \tau(K) \leq \frac{\omega(S^{n-1})}{\omega(K \cap S^{n-1})},
\]
called the \defi{volume bound} for~$K$.  Since~$\tau(K)$ is integer, we can
actually take the floor of the right side above.  It is useful, however, to call
``volume bound'' the quantity above, instead of its floor.

The volume bound provides a nontrivial upper bound, but usually a weak one.  For
instance, let~$X \subseteq S^2$ be a spherical cap with angular radius~$\pi/6$
and let~$K = \Rplus X$, so~$\tau_3 = \tau(K)$.  The volume bound provides the
upper bound
\[
  \tau_3 \leq 14.9282\ldots.
\]

If~$T$ is a regular tetrahedron with a vertex at the origin and~$K = \Rplus T$,
then~$\tau(K)$ is the tetrahedron kissing number, that is, the maximum number of
nonoverlapping regular tetrahedra that share a vertex in common.  The regular
icosahedron inscribed in the unit sphere has~$20$ triangular faces and its side
length is slightly larger than~$1$.  The vertices of each face, together with
the origin, define~$20$ pyramids, which can be shrunk to give~20 regular
tetrahedra.  This shows that~$\tau(K) \geq 20$. The volume bound gives
\[
  \tau(K) \leq 22.7946\ldots,
\]
and so~$\tau(K)$ is either~$20$, $21$, or~$22$.

It is hard to imagine how to fit a~13th sphere around the central sphere or
a~21st tetrahedron around the origin (see Figure~\ref{fig:packings}). According
to the volume bound, however, one still has enough space to squeeze in two
spheres or two tetrahedra.


\section{Characterizing intersection of spherical polygons}%
\label{sec:inter-ch}

Both the manifold optimization method to find optimal packings and the
semidefinite programming approach to compute upper bounds work for polygons for
which intersection can be easily characterized in algebraic terms.  These
polygons are called admissible (recall the definition of regular spherical
polygon from~\S\ref{sec:notation}):

\begin{definition}%
  \label{def:admissible}
  A regular spherical polygon~$K$ is \defi{admissible} if for every polygon~$K'$
  congruent to~$K$ we have~$K^\circ \cap (K')^\circ = \emptyset$ if and only if
  at least one of the following holds:
  \begin{enumerate}
    \item[(i)] the circumcircles of~$K$ and~$K'$ span interior-disjoint cones;

    \item[(ii)] a facet-inducing inequality of~$K$ separates~$K$ and~$K'$;

    \item[(iii)] a facet-inducing inequality of~$K'$ separates~$K$ and~$K'$.
  \end{enumerate}
\end{definition}

In this section, we will prove the following theorem.

\begin{theorem}%
  \label{thm:admissible-N4}
  Every regular $N$-gon for~$N \geq 4$ is admissible.
\end{theorem}

The case of triangles is more involved; a proof of the following theorem can be
found in the upcoming thesis of Spomer~\cite{Spomer2026}.

\begin{theorem}%
  \label{thm:admissible-andreas}
  A regular $N$-gon is admissible if~$N \geq 4$ or if, when~$N = 3$, its polar
  has inner angle at least~$\pi/2$ and side length at least~$\arccos(-1/4)$.
\end{theorem}


\subsection{Polarity, separation, and intersection}

The \defi{polar} of a convex cone $K \subseteq \R^n$ is the set
\[
  K^\polar = \{\, y \in \R^n : x^\tp y \leq 0\text{ for all~$x \in K$}\,\}.
\]
The cone~$K$ is \defi{polyhedral} if~$K$ is a polyhedron or, equivalently,
if~$K$ is the conic hull of finitely many vectors.  For a proper polyhedral
cone~$K$, the extreme rays of~$K^\polar$ give the facet-inducing inequalities
of~$K$ and vice versa.

Let~$K$, $K' \subseteq \R^n$ be two proper cones.  The interiors of~$K$
and~$K'$ are open convex sets, and so they are disjoint if and only if they are
separated by a hyperplane, which must pass through the origin.  In other words,
we have~$K^\circ \cap (K')^\circ = \emptyset$ if and only if there is a
nonzero~$a \in \R^n$ such that~$a^\tp x \leq 0$ for all~$x \in K$ and~$a^\tp x
\geq 0$ for all~$x \in K'$.  Using polarity, this translates to:
\[
  K^\circ \cap (K')^\circ = \emptyset\qquad\iff\qquad
  K^\polar \cap -(K')^\polar \neq \{0\}.
\]
So polarity translates separation to intersection and vice versa.  Note moreover
that, if~$K$ is a regular $N$-gon, then~$K^\polar$ is a regular $N$-gon
and~$-K^\polar$ is congruent to~$K^\polar$.  We will prove the following
theorem.

\begin{theorem}%
  \label{thm:main-intersection}
  Let~$K$ and~$K'$ be congruent regular $N$-gons for some~$N \geq 4$. We have~$K
  \cap K' \neq \{0\}$ if and only if at least one of the following holds:
  \begin{enumerate}
    \item[(i)] the barycenter of~$K$ is in~$K'$;

    \item[(ii)] a vertex of~$K$ is in~$K'$;

    \item[(iii)] a vertex of~$K'$ is in~$K$.
  \end{enumerate}
\end{theorem}

By polarity, the theorem above implies:

\begin{theorem}%
  \label{thm:main-separation}
  Let~$K$ and~$K'$ be congruent regular $N$-gons for some~$N \geq 4$ and let~$e$
  be the barycenter of~$K$.  We have~$K^\circ \cap (K')^\circ = \emptyset$ if
  and only if at least one of the following holds:
  \begin{enumerate}
    \item[(i)] the inequality~$e^\tp x \leq 0$ separates~$K$ and~$K'$;

    \item[(ii)] a facet-inducing inequality of~$K$ separates~$K$ and~$K'$;

    \item[(iii)] a facet-inducing inequality of~$K'$ separates~$K$ and~$K'$.
  \end{enumerate}
\end{theorem}

\begin{proof}
From Theorem~\ref{thm:main-intersection} by polarity.
\end{proof}

We will see in~\S\ref{sec:circumcircle-condition} that~(i) in
Theorem~\ref{thm:main-separation} is equivalent to~(i) in
Definition~\ref{def:admissible}, thus implying Theorem~\ref{thm:admissible-N4}.

Before giving a proof of Theorem~\ref{thm:main-intersection}, let us discuss the
conditions imposed on the polygon.  In the context of packing problems,
Mascarenhas and Birgin~\cite{MascarenhasB2010} characterized intersection of
polygons on~$\R^2$ by means of sentinels; see also Birgin, Martínez,
Mascarenhas, and Ronconi~\cite{BirginMMR2006} for a direct application to
packing in convex regions.  Here is an overview of their approach.

Let~$P$ be a (not necessarily convex) polygon.  A \defi{sentinel set} for~$P$ is
a subset~$S$ of~$P$ such that for every Euclidean motion~$T$ (a translation plus
a rotation) we have that~$P\cap TP \neq \emptyset$ if and only if~$TP$ contains
a point in~$S$ or~$P$ contains a point in~$TS$.

Not every polygon admits a finite sentinel set.  On the one hand, one can make
triangles intersect close to a vertex in such a way as to avoid any finite
subset of points in either triangle.  On the other hand, a finite sentinel set
for a regular polygon with at least four sides consists of its vertices plus its
barycenter.  More generally, Mascarenhas and Birgin constructed finite sentinel
sets for families of polygons whose inner angles are at least~$\pi/2$.

Theorem~\ref{thm:main-intersection} gives a sentinel set for regular spherical
polygons, in direct analogy to the planar case.  If~$N \geq 4$, then the inner
angle of the $N$-gon is at least~$\pi/2$.  As in the planar case, if~$N = 3$ and
the inner angle is less than~$\pi / 2$, then no finite sentinel set exists.
Contrary to the planar case, spherical triangles can have inner angles~$\geq
\pi/2$.  Theorem~\ref{thm:main-intersection} seems to hold as soon as the inner
angle of the polygon is~$\geq \pi/2$.  However, the proof of
Theorem~\ref{thm:main-intersection} for~$N = 3$ found in Spomer's
thesis~\cite{Spomer2026} uses an extra condition on the side length of the
polygon, as in Theorem~\ref{thm:admissible-andreas}.  It is unclear how this
condition can be removed.  Even so, after a comprehensive search for a
counterexample, it seems that the following conjecture is true (note that
necessity is clear from the discussion above).

\begin{conjecture}%
  \label{conj:admissible}
  A regular spherical polygon is admissible if and only if the inner angle of
  its polar is at least~$\pi/2$.
\end{conjecture}

The proof of Theorem~\ref{thm:main-intersection} relies on ideas from
Mascarenhas and Birgin adapted to spherical geometry.  This makes the argument
significantly more complicated.


\subsection{Proof of Theorem~\ref{thm:main-intersection}}

Given points~$x$, $y \in S^{n-1}$ such that~$\angdst(x, y) \neq \pi$, we denote
by~$[x, y]$ the geodesic between~$x$ and~$y$, that is,
\[
  [x, y] = \{\, (\lambda x + (1 - \lambda) y) / \|\lambda x + (1 - \lambda) y\|
  : \lambda \in [0, 1]\,\}.
\]
We see a regular $N$-gon both as a cone in~$\R^3$ and as a polygon on the
sphere.  We also identify facets of the cone with sides of the spherical
polygon.

Consider a regular $N$-gon~$K$ with vertices~$u_0$, \dots,~$u_{N-1}$ and side
length~$\delta$.  In this section, when dealing with vertices or facets, we
take all indices modulo~$N$.  For~$i = 0$, \dots,~$N - 1$, let
\begin{multline*}
  F_i = \{\, x \in K : \text{there are $u \in [u_{i-1}, u_i]$ and~$v \in
  [u_i, u_{i+1}]$}\\
  \text{with~$\angdst(u, v) < \delta$ such that~$x \in [u, v]$}\,\}.
\end{multline*}
The \defi{forbidden region} of~$K$ is the set~$F_0 \cup \cdots \cup F_{N-1}$.

\begin{lemma}%
  \label{lem:forbidden-triangle}
  Let~$K$ be a regular $N$-gon of side length~$\delta$ and let~$u_0$, $u_1$,
  $u_2$ be three consecutive vertices of~$K$.  If~$x \in [u_0, u_1]$ and~$y \in
  [u_1, u_2]$ are such that~$\angdst(x, y) < \delta$, then the spherical
  triangle with vertices~$u_1$, $x$, and~$y$ is contained in the forbidden
  region of~$K$.
\end{lemma}

\begin{proof}
Both~$x$ and~$u_1$ are in the spherical cap of radius~$\delta$ centered at~$y$,
and hence the complete triangle~$T$ with vertices~$u_1$, $x$, and~$y$ is
contained in this spherical cap.  In particular, the segment~$[u_1, x]$ is
contained in this spherical cap, and so every point in this segment is at
angular distance less than~$\delta$ to~$y$.  Since
\[
  T = \{\, u \in K : \text{there is $x' \in [u_1, x]$ with~$u \in [x', y]$}\,\},
\]
the result then follows.
\end{proof}

Admissibility is related to the forbidden region; in the~$N \geq 4$ case we
have:

\begin{lemma}%
  \label{lem:forbidden-admissible}
  If~$N \geq 4$, then the barycenter of a regular $N$-gon~$K$ does not belong to
  the forbidden region.
\end{lemma}

\begin{proof}
Let~$u_0$, \dots,~$u_{N-1}$ be the vertices of~$K$, let~$\delta$ be its side
length, and~$e$ be its barycenter.

We know that~$F_i$ is contained in the spherical triangle with
vertices~$u_{i-1}$, $u_i$, and~$u_{i+1}$.  For~$N \geq 5$ the barycenter~$e$ is
not in any of these triangles, and so is not in the forbidden region.  For~$N =
4$, we have that~$e$ belongs to the diagonals~$[u_0, u_2]$ and~$[u_1, u_3]$, but
the lengths of these diagonals are at least~$\delta$, and so again~$e$ is not in
the forbidden region.
\end{proof}

For this lemma to hold in the~$N = 3$ case, the side length of the polygon must
be at least~$\arccos(-1/4)$, hence the condition in
Theorem~\ref{thm:admissible-andreas}.

We are ready to prove Theorem~\ref{thm:main-intersection}.

\begin{proof}[Proof of Theorem~\ref{thm:main-intersection}]
Let~$e$ be the barycenter of~$K$.  Let~$u_0$, \dots,~$u_{N-1}$ be the vertices
of~$K$ and~$a_0$, \dots,~$a_{N-1}$ be the facet normals, in such a way
that~$a_i$ is the normal of the hyperplane going through the origin, $u_i$,
and~$u_{i+1}$ and
\[
  K = \{\, x \in \R^3 : a_i^\tp x \leq 0\text{ for $i = 0$, \dots,~$N-1$}\,\}.
\]
Define~$u'_i$ and~$a'_i$ similarly for~$K'$.

By contradiction, assume that~$K \cap K' \neq \{0\}$ but that none of the
conditions (i)--(iii) hold.  So there must be a point~$z \in S^{n-1}$ in the
interior of~$K \cap K'$.  Since~$e \notin K'$, the segment~$[e, z]$ must cross a
side of~$K'$, say~$[u'_0, u'_1]$, at exactly one point~$v$; note that~$v \in K$
as well.

Since~$v \in K$ and since~$u'_0$, $u'_1 \notin K$, the segment~$[u'_0, u'_1]$
must intersect two sides of~$K$, one at a point~$x$ and the other at a
point~$y$, and so the angular distance between these two sides is at
most~$\angdst(x, y) < \delta$, where~$\delta$ is the side length of~$K$
and~$K'$.  Since the angular distance between two disjoint sides of~$K$ is at
least~$\delta$, we see that the two sides intersected by~$[u'_0, u'_1]$ must
share a vertex, and so there is~$i$ such that~$x \in [u_{i-1}, u_i]$ and~$y \in
[u_i, u_{i+1}]$.

Let~$T$ be the spherical triangle with vertices~$u_i$, $x$, and~$y$.  Depending
on which side of the facet~$(a'_0)^\tp x \leq 0$ the vertex~$u_i$ is, we either
have~$e \in T$ or~$z \in T$.

If~$e \in T$, then since~$\angdst(x, y) < \delta$, we see from
Lemma~\ref{lem:forbidden-triangle} that~$e$ belongs to the forbidden region
of~$K$ and with Lemma~\ref{lem:forbidden-admissible} we get a contradiction.

So assume~$z \in T$.  Since~$u_i \notin K'$, there is a facet of~$K'$ that
separates~$u_i$.  Since $z \in T$, this facet intersects the segments~$[x, u_i]$
and~$[y, u_i]$, and so its distance to~$[u'_0, u'_1]$ is less than~$\delta$.
Again, this implies that this facet shares a vertex with~$[u'_0, u'_1]$; assume
without loss of generality that the facet is~$[u'_{-1}, u'_0]$.

Now change the coordinate system so that~$a_0' = (0, -1, 0)$ and~$x = (0, 0,
1)$.  Since the inner angles are at least~$\pi/2$, and since~$[x, u_i]$
intersects~$[u'_{-1}, u'_0]$, the angle at~$x$ between~$[u'_1, x]$ and~$[x,
u_i]$ is at least~$\pi/2$.  This means that we can first move inside of~$K$
from~$x$ perpendicular to~$[u'_0, u'_1]$ until reaching a point~$x'$ at the
boundary of~$K$.  Such a movement is achieved by applying a rotation that fixes
the~$(1, 0, 0)$-axis, that is, there is~$\alpha > 0$ such that~$x' = R_X(\alpha)
x$.

Next, since the inner angle of~$K$ at~$u_i$ is also at least~$\pi/2$, we can
continue moving from~$x'$ parallel to~$[u'_0, u'_1]$ without leaving~$K$ until
we intersect the segment~$[u'_0, u'_1]$.  This intersection will be a point~$y'
\in [x, y]$.

By construction, the plane defined by~$x'$, $y'$, and~$0$ is obtained from the
plane defined by~$u'_0$, $u'_1$, and~$0$ by a rotation~$R_X$.  This implies
that~$y'$ lies at the equator, and hence~$\angdst(x, y') = \pi/2$.  So we
get~$\delta > \angdst(x, y) \geq \pi/2$, a contradiction since the side length
of a regular $N$-gon for~$N \geq 4$ is always less than~$\pi/2$.
\end{proof}

The proof above fails for~$N = 3$ since spherical triangles can have side length
greater than~$\pi/2$.  In this case, the same reasoning can be carried out, with
slight changes, to again get a contradiction, but it is necessary to relate the
inner angle to the polygon's side length, complicating the
argument~\cite{Spomer2026}.


\subsection{The circumcircle condition}%
\label{sec:circumcircle-condition}

The cone spanned by the incircle of the polar of a regular $N$-gon~$K$ is the
polar of the cone spanned by the circumcircle of~$K$ and vice versa.  So
Theorem~\ref{thm:admissible-N4} follows, via polarity, from
Theorem~\ref{thm:main-intersection} once condition~(i) of the latter is replaced
by
\begin{enumerate}
  \item[(${\rm i}'$)] the incircles of~$K$ and~$K'$ intersect.
\end{enumerate}

\begin{proof}[Proof of Theorem~\ref{thm:admissible-N4}]
We prove Theorem~\ref{thm:main-intersection} with~(${\rm i}'$) instead of~(i).

Let~$K$ be the standard regular $N$-gon with angular radius~$\varrho$.
Let~$u_0$, \dots,~$u_{N-1}$ be the vertices of~$K$ and let~$K'$ be a rotation
of~$K$.  To prove Theorem~\ref{thm:main-intersection} with~(${\rm i}'$) instead
of~(i), we show that if the incircles of~$K$ and~$K'$ are disjoint, then the
barycenter of~$K$ is not in~$K'$ and vice versa; the alternative
Theorem~\ref{thm:main-intersection} then follows from the original.

The incircle of~$K$ is centered at its barycenter~$e_3$ and touches the
midpoint~$y = (u_0+u_1)/\|u_0+u_1\|$ of the~$[u_0, u_1]$ side of the polygon.
Let~$\alpha = 2\pi/N$ and let~$\tau$ be the inradius of~$K$.  With~$\mu =
\|u_0+u_1\|$, we have
\[
  \mu^2 = 2 + 2\cos\alpha - 2\cos\alpha\cos^2\varrho + 2\cos^2\varrho
\]
and
\[
  \tau = \arccos e_3^\tp y = \arccos \frac{2\cos\varrho}{\mu}.
\]

Let~$\beta$ be the angular distance between the barycenters of~$K$ and~$K'$.
Assume that the incircles of~$K$ and~$K'$ are disjoint, that is,~$\beta >
2\tau$.  We claim that~$\beta > \varrho$, which implies that the barycenter
of~$K$ is not in~$K'$ and vice versa, as wanted.  To show that~$\beta >
\varrho$, we show that~$2\tau \geq \varrho$, since as~$\beta > 2\tau$ we then
get~$\beta > \varrho$.  Showing that~$2\tau \geq \varrho$ is equivalent to
showing that~$\cos(2\tau) \leq \cos\varrho$.

Since
\[
  \cos(2\tau) = 2\cos^2\tau - 1\qquad\text{and}\qquad
  \cos^2\tau = \frac{4\cos^2\varrho}{\mu^2},
\]
the inequality~$\cos(2\tau) \leq \cos\varrho$ is equivalent to
\[
  0 \leq \cos\varrho - 2\cos^2\tau + 1 =
  \cos\varrho - \frac{8\cos^2\varrho}{\mu^2} + 1.
\]
Multiply both sides by~$\mu^2$ to get the equivalent inequality
\begin{equation}%
  \label{eq:goal-ineq}
  \mu^2\cos\varrho - 8\cos^2\varrho + \mu^2 \geq 0.
\end{equation}
With~$u = \cos\alpha$ and~$v = \cos\varrho$, the left side above is a polynomial
on~$u$ and~$v$, namely
\[
  p(u, v) = -2uv^3 + 2v^3 - 2uv^2 - 6v^2 + 2uv + 2v + 2u + 2.
\]

Inequality~\eqref{eq:goal-ineq} holds for all~$\alpha \in [0, 2\pi/3]$
and~$\varrho \in [0, \pi/2]$ if and only if~$p(u, v) \geq 0$ for all~$u \in
[-1/2, 1]$ and~$v \in [0, 1]$.  Since the only roots of
\[
  \frac{\partial p}{\partial u} = -2v^3 - 2v^2 + 2v + 2
\]
are~$\pm 1$, the polynomial~$p$ does not have a critical point, and hence a
local minimum, in the interior of the domain~$[-1/2, 1] \times [0, 1]$.  It then
remains to show that~$p$ is nonnegative on the boundary of the domain as well.

We easily check that~$p(u, 1) = 0$ and~$p(u, 0) = 2u + 2$, which are nonnegative
in the domain.  Likewise,~$p(-1/2, v) = 3v^3 - 5v^2 + v + 1$ and~$p(1, v) =
-8v^2 + 4v + 4$, and computing the roots of these polynomials we check that they
are also nonnegative in the domain, finishing the proof.
\end{proof}


\section{Lower bounds via manifold optimization}%
\label{sec:manifold-opt}

Let~$K_\varrho$ be the standard regular $N$-gon with angular radius~$\varrho$.
Given an integer~$n \geq 1$, our goal is to maximize~$\varrho$ such that there
is a packing of~$n$ copies of~$K_\varrho$.  Finding a packing of~$n$ polygons is
the same as finding~$n$ appropriate matrices in~$\sorto(3)$ and, by rotation
invariance, we can assume that one of the matrices is the identity.  This leads
us to the optimization problem
\begin{equation}%
  \label{opt:manifold-first}
  \begin{optprob}
    \text{maximize}&\varrho\\
    &K_\varrho^\circ \cap S_i K_\varrho^\circ = \emptyset&\text{for $1 \leq i
    \leq n-1$,}\\
    &S_i K_\varrho^\circ \cap S_j K_\varrho^\circ = \emptyset&\text{for $1 \leq
    i < j \leq n-1$,}\\
    &\onerow{S_1, \ldots, S_{n-1} \in \sorto(3),\ \varrho \in (0, \pi/2).}
  \end{optprob}
\end{equation}

This is an optimization problem on the smooth Riemannian manifold~$(0, \pi/2)
\times \sorto(3)^{n-1}$, and good solutions for it can be found via methods of
manifold optimization.  For this, the constraints have to be rephrased, and so
we assume that~$K$ is admissible (Definition~\ref{def:admissible}).

Let~$u_0(\varrho)$, \dots,~$u_{N-1}(\varrho)$ be the vertices
and~$a_0(\varrho)^\tp x \leq 0$, \dots,~$a_{N-1}(\varrho)^\tp x \leq 0$ be the
facet-inducing inequalities of~$K_\varrho$.  For~$S \in \sorto(3)$, the vertices
of~$SK_\varrho$ are~$S u_i(\varrho)$ and the facet-inducing inequalities
are~$(Sa_i(\varrho))^\tp x \leq 0$.

Consider the functions~$p$, $q_{ij}\colon (0, \pi/2) \times \sorto(3) \to \R$
such that
\begin{align*}
  p(\varrho, S) &= \cos(2\varrho) - S_{33}\quad\text{and}\\
  q_{ij}(\varrho, S) &= a_i(\varrho)^\tp S u_j(\varrho).
\end{align*}
It follows from the admissibility of~$K$ that, for~$\varrho \in (0, \pi/2)$
and~$S$, $T \in \sorto(3)$, we have~$SK_\varrho^\circ \cap TK_\varrho^\circ =
\emptyset$ if and only if~$(S, T)$ is in one of the sets
\begin{align*}
  \Sigma_1(\varrho) &= \{\, (S, T) \in \sorto(3)^2 : p(\varrho, S^\tp T) \geq 0\,\},\\
  \Sigma_2(\varrho) &= \bigcup_{i=0}^{N-1} \bigcap_{j=0}^{N-1} \{\, (S, T) \in
  \sorto(3)^2 : q_{ij}(\varrho, S^\tp T) \geq 0\,\},\quad\text{or}\\
  \Sigma_3(\varrho) &= \bigcup_{i=0}^{N-1} \bigcap_{j=0}^{N-1} \{\, (S, T) \in
  \sorto(3)^2 : q_{ij}(\varrho, T^\tp S) \geq 0\,\}.
\end{align*}
With this,~\eqref{opt:manifold-first} becomes
\begin{equation}%
  \label{opt:manifold-second}
  \begin{optprob}
    \text{maximize}&\varrho\\
    &(I, S_i) \in \Sigma_1(\varrho) \cup \Sigma_2(\varrho) \cup
    \Sigma_3(\varrho)&\text{for~$1 \leq i \leq n-1$,}\\
    &(S_i, S_j) \in \Sigma_1(\varrho) \cup \Sigma_2(\varrho) \cup
    \Sigma_3(\varrho)&\text{for~$1 \leq i < j \leq n-1$,}\\
    &\onerow{S_1, \ldots, S_{n-1} \in \sorto(3),\ \varrho \in (0, \pi/2).}
  \end{optprob}
\end{equation}

The augmented Lagrangian method (see, e.g., Birgin and
Martínez~\cite{BirginM2014}) is an algorithm for constrained optimization
problems like
\begin{equation}
  \label{opt:manifold-standard}
  \begin{optprob}
    \text{minimize}&f(x)\\
    &g_i(x) \leq 0&\text{for $i = 1$, \dots,~$m$,}\\
    &x \in \Mcal,
  \end{optprob}
\end{equation}
where~$\Mcal = \R^n$ and the functions~$f$, $g_i\colon\Mcal \to \R$ are twice
continuously differentiable.  This method has been generalized by Liu and
Boumal~\cite{LiuB2020} to the case when~$\Mcal$ is a Riemannian manifold.

Given a penalty parameter~$\tau > 0$ and multipliers~$\lambda \in \R^m$,
$\lambda \geq 0$, we consider the augmented Lagrangian function
\[
  L_\tau(x, \lambda) = f(x) + (\tau/2) \sum_{i=1}^m \max\{0, \lambda_i/\tau +
  g_i(x)\}^2.
\]
For fixed~$\tau$ and~$\lambda$, a first-order stationary point of~$x \mapsto
L_\tau(x, \lambda)$ over~$\Mcal$ can be computed using standard techniques from
smooth Riemannian optimization, like gradient descent or Newton-type methods;
the book by Boumal~\cite{Boumal2023} gives an introduction to optimization on
manifolds.

After a stationary point is found, the penalty parameter and Lagrangian
multipliers are updated based on the constraint violations.  Repeating this
procedure leads to an approximate stationary point of~$f$ on the manifold
satisfying the given constraints; for convergence results, see Liu and
Boumal~\cite{LiuB2020}.

Problem~\eqref{opt:manifold-second} is not quite in the format
of~\eqref{opt:manifold-standard}, however, since it involves constraints of the
form
\[
  g_1(x) \leq 0\quad\text{or}\quad g_2(x) \leq 0.
\]
Such a constraint can be modeled by requiring that~$g(x) = \min\{g_1(x),
g_2(x)\} \leq 0$.  The function~$g$ is in general not differentiable, though
from Rademacher's theorem we know it is differentiable almost everywhere.  In
practice, the nondifferentiability of~$g$ does not cause any problems for the
algorithm.

After finding an approximate stationary point~$(\varrho, S_1, \ldots, S_{n-1})$,
we can round~$\varrho$ and the entries of the matrices~$S_i$ to rational numbers
and then check that the resulting configuration is indeed a packing.  To round
the matrices, we use the Cayley transform~$\cayley(S) = (I - S)(I+S)^{-1}$, a
self-inverse bijection between rotation matrices without~$-1$ eigenvalues and
skew-symmetric matrices.  To round the matrix~$S$, we approximate~$\cayley(S)$
by a rational matrix~$A$; the rational approximation of~$S$ is then~$\cayley((A
- A^\tp)/2) \in \sorto(3)$.

A few optimal packings of spherical caps are known, and the manifold
optimization method described above can be used to approximate these optimal
packings quite closely (see Table~\ref{tab:caps} and the discussion
in~\S\ref{sec:results}).  These benchmark results indicate that the manifold
optimization method, though a numerical method, is able to find near-optimal
packings.

We used this method to compute several packings of spherical polygons.  All
results are listed in Table~\ref{tab:bounds} and discussed
in~\S\ref{sec:results}.  Some of the configurations computed are depicted in
Figures~\ref{fig:maximal-packings} and~\ref{fig:large-packings}.


\section{The Lovász theta number}%
\label{sec:theta}

The Lovász theta number~\cite{Lovasz1979} is an upper bound for the independence
number of a graph in terms of semidefinite programming.  It is the framework
behind several bounds in discrete geometry, like the linear programming bound of
Delsarte, Goethals, and Seidel~\cite{DelsarteGS1977} for spherical codes, the
linear programming bound of Cohn and Elkies~\cite{CohnE2003} for sphere-packing
density, and the lower bound for the measurable chromatic number of the
Euclidean space of Oliveira and Vallentin~\cite{OliveiraV2010}, to cite but a
few applications.  It is also the framework we will use to bound the packing
number of a proper cone.

Let~$G = (V, E)$ be a graph.  A subset of~$V$ is \defi{independent} if it does
not contain an edge.  The \defi{independence number} of~$G$, denoted
by~$\alpha(G)$, is the maximum cardinality of an independent set of~$G$.

Given a proper cone~$K \subseteq \R^n$, consider the graph~$G(K)$ whose vertex
set is~$\sorto(n)$ and in which
\[
  \text{$S$, $T \in \sorto(n)$ are adjacent}\qquad\text{if}\qquad S K^\circ
  \cap T K^\circ \neq \emptyset.
\]
Note that~$G(K)$ is the Cayley graph with connection set
\[
  \{\, S \in \sorto(n) : K^\circ \cap SK^\circ \neq \emptyset\,\}.
\]

If~$I\subseteq\sorto(n)$ is an independent set of~$G(K)$, then~$\{\, SK : S \in
I\,\}$ is a packing of~$K$.  Conversely, a packing of~$K$ corresponds to an
independent set of~$G(K)$, and so~$\tau(K) = \alpha(G(K))$.

The upper bound for the independence number of~$G(K)$ is given by the following
optimization problem, whose optimal value is denoted by~$\vartheta'(G(K))$:
\begin{equation}%
  \label{opt:theta-kernel}
  \begin{optprob}
    \text{minimize}&\lambda\\
    &A(S, S) \leq \lambda - 1&\text{for all~$S \in \sorto(n)$,}\\
    &A(S, T) \leq -1&\text{if~$SK^\circ \cap TK^\circ = \emptyset$,}\\
    &\onerow{\text{$A\colon \sorto(n)^2 \to \R$ is continuous and positive semidefinite.}}
  \end{optprob}
\end{equation}

Here, the optimization variable is the continuous kernel~$A$.  The kernel~$A$ is
\defi{positive semidefinite} if for every finite set~$\Ucal \subseteq \sorto(n)$
the matrix~$\bigl(A(S, T)\bigr)_{S, T \in \Ucal}$ is positive semidefinite, that
is, if every finite principal submatrix of~$A$ is positive semidefinite.  This
implies in particular that~$A$ is symmetric.  The notation~$\vartheta'$ for the
optimal value of~\eqref{opt:theta-kernel} is appropriate: this is an extension
of the theta prime number of McEliece, Rodemich, and
Rumsey~\cite{McElieceRR1978} and Schrijver~\cite{Schrijver1979}, itself a
strengthening of the Lovász theta number, to the infinite graph~$G(K)$.

\begin{theorem}%
  \label{thm:theta-bound}
  If~$(\lambda, A)$ is a feasible solution of~\eqref{opt:theta-kernel},
  then~$\tau(K) \leq \lambda$.  Moreover,~$\vartheta'(G(K))$ is at most the
  volume bound for~$K$.
\end{theorem}

\begin{proof}
Let~$I \subseteq \sorto(n)$ be a nonempty independent set of~$G(K)$ and
let~$(\lambda, A)$ be a feasible solution of~\eqref{opt:theta-kernel}.  For
distinct~$S$, $T \in I$ we have~$SK^\circ \cap TK^\circ = \emptyset$,
hence~$A(S, T) \leq -1$.  It follows that
\[
  0 \leq \sum_{S, T \in I} A(S, T) \leq (\lambda - 1) |I| - (|I|^2 - |I|),
\]
hence~$|I| \leq \lambda$, proving the first part of the statement.

For the comparison with the volume bound, let~$\mu$ denote the Haar probability
measure on~$\sorto(n)$.  Fix a point~$e \in S^{n-1}$ and define the set~$\Xcal =
\{\, S \in \sorto(n) : Se \in K^\circ\,\}$. The surface measure on~$S^{n-1}$ is
the pushforward of the Haar measure through the map $S \mapsto Se$, and so
\[
  \mu(\Xcal) = \frac{\omega(K \cap S^{n-1})}{\omega(S^{n-1})}.
\]
In particular, the volume bound is~$1 / \mu(\Xcal)$.

Set
\[
  A'(S, T) = \mu(S\Xcal \cap T\Xcal) = \int_{\sorto(n)} \indf{S\Xcal}(R)
  \indf{T\Xcal}(R)\, d\mu(R),
\]
where~$\one_\Xcal$ is the indicator function of~$\Xcal$.

A direct proof that~$A'$ is continuous is not hard; continuity also follows by
seeing~$A'(S, T)$ as the result of a convolution of
functions~\cite[Proposition~2.1]{Folland1995}.  By construction,~$A'$ is
positive semidefinite, since a principal submatrix of~$A'$ indexed by a finite
set~$\Ucal \subseteq \sorto(n)$ is given above as the Gram matrix of the
vectors~$\indf{S\Xcal}$ for~$S \in \Ucal$ with respect to the standard inner
product on~$L^2(\sorto(n))$.  We also have~$A'(S, T) = 0$ whenever~$SK^\circ
\cap TK^\circ = \emptyset$, since if the latter condition holds then~$S\Xcal
\cap T\Xcal = \emptyset$.

Now set~$A = \mu(\Xcal)^{-2} A' - J$, where~$J$ is the constant-one kernel. Note
that~$A$ is continuous and that~$A(S, T) = -1$ if~$SK^\circ \cap TK^\circ =
\emptyset$. We claim that~$A$ is positive semidefinite.  Indeed, the
constant-one function~$\one$ is an eigenfunction of~$J$ with eigenvalue~$1$.  It
is also an eigenfunction of~$A'$ with eigenvalue~$\mu(\Xcal)^2$, since for
every~$S \in \sorto(n)$ we have
\[
  \begin{split}
    \int_{\sorto(n)} A'(S, T)\, d\mu(T) &= \int_{\sorto(n)}\int_{\sorto(n)}
    \indf{S\Xcal}(R) \indf{T\Xcal}(R)\, d\mu(R) d\mu(T)\\
    &=\int_{\sorto(n)}\int_{\sorto(n)} \indf{\Xcal}(S^{-1} R)
    \indf{\Xcal}(T^{-1} R)\, d\mu(R) d\mu(T)\\
    &=\int_{\sorto(n)} \indf{\Xcal}(S^{-1} R) \int_{\sorto(n)}
    \indf{\Xcal}(T^{-1} R)\, d\mu(T) d\mu(R)\\
    &=\int_{\sorto(n)} \indf{\Xcal}(S^{-1} R) \mu(\Xcal)\, d\mu(R)\\
    &=\mu(\Xcal)^2.
  \end{split}
\]
Since~$\one$ is, up to scaling, the only eigenfunction of~$J$ with nonzero
eigenvalue, it follows from the spectral theorem that~$A$ is positive
semidefinite.

Finally, for~$S \in \sorto(n)$ we have
\[
  A(S, S) = \mu(\Xcal)^{-2} \mu(\Xcal) - 1 = \mu(\Xcal)^{-1} - 1
\]
and so, by taking~$\lambda = \mu(\Xcal)^{-1}$, it follows that~$(\lambda, A)$
is a feasible solution of~\eqref{opt:theta-kernel}, hence~$\vartheta'(G(K))
\leq \mu(\Xcal)^{-1}$, as we wanted.
\end{proof}

The volume bound as defined above is a special case of a more general
construction explored in~\S\ref{sec:volume-bound}.  It turns out that the theta
number gives a better bound than any of these generalized volume bounds.

Any~$R \in \sorto(n)$ gives an automorphism of the graph~$G(K)$, since~$S$, $T
\in \sorto(n)$ are adjacent if and only if~$RS$, $RT$ are.  These symmetries
lead to a first simplification of~\eqref{opt:theta-kernel}.

A matrix~$R \in \sorto(n)$ acts on a kernel~$A\colon \sorto(n)^2 \to \R$ by
\[
  (R\cdot A)(S, T) = A(R^{-1} S, R^{-1} T)\qquad\text{for~$S$, $T \in
  \sorto(n)$.}
\]
The kernel~$A$ is \defi{invariant} if~$R\cdot A = A$ for all~$R \in \sorto(n)$.
If in~\eqref{opt:theta-kernel} we restrict ourselves to invariant kernels, then
we still get an upper bound for~$\tau(K)$.  Moreover, such a restriction does
not make the bound worse.  Indeed, if~$(\lambda, A)$ is a feasible solution
of~\eqref{opt:theta-kernel}, then~$(\lambda, \overline{A})$ with
\[
  \overline{A}(S, T) = \int_{\sorto(n)} A(RS, RT)\, d\mu(R),
\]
where~$\mu$ is the Haar probability measure on~$\sorto(n)$, is also a feasible
solution and~$\overline{A}$ is invariant.

If~$A$ is an invariant kernel, then~$A(S, T) = A(T^{-1}S, I)$ for all~$S$, $T
\in \sorto(n)$, that is, the value of~$A(S, T)$ depends only on~$T^{-1}S$.  So
there is a continuous function~$f\colon \sorto(n) \to \R$ such that~$A(S, T) =
f(T^{-1}S)$.  Conversely, such a continuous function~$f$ defines a continuous
invariant kernel~$A$ by the same identity.

A continuous function~$f\colon \sorto(n) \to \R$ is of \defi{positive type}
if~$f(S^{-1}) = f(S)$ for all~$S \in \sorto(n)$ and the kernel~$(S, T) \mapsto
f(T^{-1}S)$ is positive semidefinite.  We can rewrite~\eqref{opt:theta-kernel}
equivalently as
\begin{equation}%
  \label{opt:theta-func}
  \begin{optprob}
    \text{minimize}&1 + f(I)\\
    &f(S) \leq -1&\text{if $K^\circ \cap SK^\circ = \emptyset$,}\\
    &\onerow{\text{$f\colon \sorto(n) \to \R$ is continuous and of positive
    type.}}
  \end{optprob}
\end{equation}
This problem is closer in shape to the linear programming bounds of Delsarte,
Goethals, and Seidel~\cite{DelsarteGS1977} and Cohn and Elkies~\cite{CohnE2003}.


\subsection{Upper bounds from spherical codes}%
\label{sec:spherical-codes}

Let~$K \subseteq \R^n$ be a proper cone and let~$X$ be the spherical cap of
largest radius contained in~$K \cap S^{n-1}$; write~$C = \Rplus X$.
Then~$\tau(K) \leq \tau(C)$, and any upper bound for~$\tau(C)$ is also an upper
bound for~$\tau(K)$.

Note that~$G(C)$ is a subgraph of~$G(K)$.  Any feasible solution
of~\eqref{opt:theta-kernel} for~$G(C)$ is also a feasible solution
of~\eqref{opt:theta-kernel} for~$G(K)$, hence~$\vartheta'(G(C)) \geq
\vartheta'(G(K))$.  However, computing~$\vartheta'(G(C))$ can be significantly
simpler than computing~$\vartheta'(G(K))$, since independent sets of~$G(C)$
correspond to spherical codes.

Indeed, let~$\phi$ be the angular radius of the spherical cap~$X$.  Independent
sets of~$G(C)$ correspond to packings of spherical caps of angular radius~$\phi$
or, in other words, to spherical codes with minimum angular distance~$2\phi$:
sets of points on~$S^{n-1}$ any two of which are at distance at least~$2\phi$
apart.  The linear programming bound of Delsarte, Goethals, and
Seidel~\cite{DelsarteGS1977} gives an upper bound for the cardinality of such a
spherical code.  Bachoc, Nebe, Oliveira, and Vallentin~\cite{BachocNOV2009}
reinterpreted the linear programming bound as the theta prime number of the
graph whose vertex set is the sphere and in which~$x$, $y \in S^{n-1}$ are
adjacent if~$\angdst(x, y) \in (0, 2\phi)$.  Namely, they showed that the linear
programming bound is the optimal value of the problem
\begin{equation}%
  \label{opt:lp-sphere}
  \begin{optprob}
    \text{minimize}&\lambda\\
    &A(x, x) \leq \lambda - 1&\text{for all~$x \in S^{n-1}$,}\\
    &A(x, y) \leq -1&\text{if $\angdst(x, y) \geq 2\phi$,}\\
    &\onerow{\text{$A\colon (S^{n-1})^2 \to \R$ is continuous and positive
    semidefinite.}}
  \end{optprob}
\end{equation}

\begin{theorem}%
  \label{thm:spherical-codes-cmp}
  The optimal value of~\eqref{opt:lp-sphere} equals~$\vartheta'(G(C))$.
\end{theorem}

\begin{proof}
Fix~$e \in S^{n-1}$.  Note that
\begin{equation}%
  \label{eq:GC-edges}
  \text{$S$, $T \in \sorto(n)$ are adjacent in~$G(C)$\qquad$\iff$\qquad}
  0 < Se \cdot Te < 2\phi.
\end{equation}

Let~$(\lambda, A)$ be a feasible solution of~\eqref{opt:lp-sphere}.  For~$S$, $T
\in \sorto(n)$, set~$A'(S, T) = A(Se, Te)$.  Note that~$A'$ is continuous and
positive semidefinite, and so using~\eqref{eq:GC-edges} we see that~$(\lambda,
A')$ is a feasible solution of~\eqref{opt:theta-kernel},
hence~$\vartheta'(G(C))$ is at most the optimal value of~\eqref{opt:lp-sphere}.

For the reverse inequality, let~$(\lambda, A)$ be a feasible solution
of~\eqref{opt:theta-kernel}.  We know that we can assume that~$A$
is invariant under~$\sorto(n)$; in particular,~$A$ has constant diagonal.  We
can also assume that
\begin{equation}%
  \label{eq:A-property}
  A(S, T) = A(S', T')\qquad\text{whenever}\qquad (Se, Te) = (S'e, T'e).
\end{equation}

Indeed, let~$\Scal = \{\, S \in \sorto(n) : Se = e\,\}$ be the stabilizer
subgroup of~$e$.  Let~$\mu$ be the Haar probability measure on~$\Scal$ and set
\begin{equation}%
  \label{eq:over-A-def}
  \overline{A}(S, T) = \int_\Scal \int_\Scal A(SR_1, TR_2)\, dR_2 dR_1.
\end{equation}
We claim that~$(\lambda, \overline{A})$ is a feasible solution
of~\eqref{opt:theta-kernel}, which by construction
satisfies~\eqref{eq:A-property}.

Note that~$\overline{A}$ is continuous.  With~\eqref{eq:GC-edges} we see
that~$\overline{A}(S, T) \leq -1$ whenever~$S$, $T$ are adjacent.  Since~$A$ is
continuous and positive semidefinite, if~$S$, $T \in \sorto(n)$, then the matrix
\[
  \begin{pmatrix}
    A(S, S)&A(S, T)\\
    A(S, T)&A(T, T)
  \end{pmatrix}
\]
is positive semidefinite, and hence has nonnegative determinant.  It follows
that $|A(S, T)| \leq \delta$, where~$\delta$ is the common value of the diagonal
entries of~$A$, hence $\overline{A}(S, S) \leq \lambda - 1$ for all~$S$.
Finally, it follows from the spectral theorem, by distributing the double
integral in~\eqref{eq:over-A-def} over a spectral decomposition of~$A$,
that~$\overline{A}$ is positive semidefinite, proving the claim.

So assume that~$A$ satisfies~\eqref{eq:A-property} and for~$x$, $y \in S^{n-1}$
let~$A'(x, y) = A(Se, Te)$, where~$S$ and~$T$ are any elements of~$\sorto(n)$
such that~$(Se, Te) = (x, y)$.  Since~$A$ satisfies~\eqref{eq:A-property}, the
kernel~$A'$ is well defined.  Using~\eqref{eq:GC-edges} it is then immediate
that~$(\lambda, A')$ is a feasible solution of~\eqref{opt:lp-sphere}, finishing
the proof of the reverse inequality.
\end{proof}

Though Theorem~\ref{thm:spherical-codes-cmp} shows that the linear programming
bound for spherical codes can never be better than the theta prime number for
the original packing graph~$G(K)$, the linear programming bound is far from
useless: indeed, it is enough to prove maximality of some configurations. Though
it can never be better than~$\vartheta'(G(K))$, it is much simpler to compute.
Moreover, there are other upper bounds for spherical codes, like for instance
the Bachoc-Vallentin~\cite{BachocV2008} bound, that are better than the linear
programming bound; these improved bounds can be better than~$\vartheta'(G(K))$.
Finally, bounds for spherical codes are often provably not sufficient to show
maximality of certain configurations; such cases are also discussed
in~\S\ref{sec:results}.


\section{Harmonic analysis on $\sorto(3)$}%
\label{sec:harmonic}

Our goal is to find upper bounds for packings of proper cones in~$\R^3$ by
computing feasible solutions of~\eqref{opt:theta-func}.  A first step in this
direction is to parameterize positive type functions~$f\colon \sorto(3) \to \R$
in some convenient way.  This parametrization uses the representation theory
of~$\sorto(3)$; for basic terminology and facts on representation theory, the
reader is referred to the book by Bröcker and tom Dieck~\cite{BrockerD1985}.

Let~$G$ be a compact group.  A continuous function~$f\colon G \to \C$
is of \defi{positive type} if~$f(x^{-1}) = \overline{f(x)}$ for all~$x \in
G$ and the kernel~$(x, y) \mapsto f(y^{-1}x)$ is positive
semidefinite, that is, if for every finite set~$U \subseteq G$ the matrix
$\bigl(f(y^{-1}x)\bigr)_{x, y \in U}$ is (Hermitian) positive
semidefinite.

The \defi{dual group}~$\widehat{G}$ of~$G$ is the set of all irreducible unitary
representations of~$G$ (which are finite dimensional) up to equivalence.  We
assume that the representation spaces are always~$\C^n$ for some~$n \geq 1$, and
so an irreducible representation in~$\widehat{G}$ is a mapping from~$G$ to the
group of~$n \times n$ unitary matrices for some~$n \geq 1$.

{\sloppy
Bochner~\cite{Bochner1941} characterized continuous functions of positive type
on compact groups (see also Folland~\cite{Folland1995}):
\par}

\begin{theorem}%
  \label{thm:Bochner}
  Let~$G$ be a compact group and let~$\widehat{G}$ be its dual group. A
  function~$f\colon G \to \C$ is continuous and of positive type if and only if
  there are Hermitian positive-semidefinite matrices~$F_\rho$ for~$\rho \in
  \widehat{G}$ with~$\sum_{\rho\in\widehat{G}} \trace F_\rho < \infty$ such that
  \[
    f(x) = \sum_{\rho \in \widehat{G}} \langle F_\rho,
    \rho(x)\rangle,
  \]
  in which case the series converges absolutely and uniformly.
\end{theorem}

The direction of Bochner's theorem that we need, namely that if~$f$ is given by
an expression as above for positive-semidefinite matrices~$F_\rho$ then it is
continuous and of positive type, is easy to prove.  Indeed, say~$f$ is given by
the expression above.  Continuity follows from the continuity of the irreducible
representations, together with the positive-semidefiniteness of~$F_\rho$ and the
condition that~$\sum_{\rho\in\widehat{G}} \trace F_\rho < \infty$.  To see
that~$f$ is of positive type, let~$U \subseteq G$ be a finite set and let~$a_x$,
for $x \in U$, be complex numbers.  For~$\rho \in \widehat{G}$ write
\[
  \begin{split}
    E_\rho = \sum_{x, y \in U} a_x \overline{a_y} \rho(y^{-1}x)
    &=\sum_{x, y \in U} a_x \overline{a_y} \rho(y)^*\rho(x)\\
    &=\biggl(\sum_{x \in U} a_x \rho(x)\biggr)^*
      \biggl(\sum_{x \in U} a_x \rho(x)\biggr);
  \end{split}
\]
note that~$E_\rho$ is positive semidefinite.  Then, since the~$F_\rho$ matrices
are positive semidefinite,
\[
  \sum_{x, y \in U} a_x \overline{a_y} f(y^{-1}x)
  = \sum_{\rho \in \widehat{G}} \langle F_\rho, E_\rho\rangle \geq 0,
\]
and so~$f$ is of positive type.

To apply Bochner's theorem to~$\sorto(3)$, we need to know its irreducible
representations.  The theory is classical; the relevant details are presented
here to keep the paper self contained.

To study the representation theory of~$\sorto(3)$, we need to fix a
parametrization of rotation matrices.  There are several alternatives, but for
us Euler angles, defined in~\S\ref{sec:notation}, is the most convenient one.


\subsection{Representation theory of~$\sorto(3)$}%
\label{sec:sorto3-reps}

The representation theory of~$\sorto(3)$ is closely related to that
of~$\suni(2)$, the group of~$2\times 2$ unitary matrices with determinant~$1$.
This section is based on expositions of the classical theory by Andrews, Askey,
and Roy~\cite{AndrewsAR1999}, Bröcker and tom Dieck~\cite{BrockerD1985}, and
Chirikjian and Kyatkin~\cite{ChirikjianK2001}.

Any~$U \in \suni(2)$ can be written as
\begin{equation}%
  \label{eq:suni-matrix}
  U = \begin{pmatrix}
    a&b\\
    -\overline{b}&\overline{a}
  \end{pmatrix}
\end{equation}
with~$|a|^2 + |b|^2 = 1$.  It follows that~$U$ is uniquely determined by~$|a|$
and the arguments of~$a$ and~$b$.  If we write
\[
  a = e^{i(\alpha + \gamma) / 2} \cos(\beta / 2)\qquad\text{and}\qquad
  b = i e^{i(\alpha-\gamma)/2} \sin(\beta/2),
\]
then~$U = U(\alpha, \beta, \gamma) = U_Z(\gamma) U_X(\beta) U_Z(\alpha)$, where
\[
  U_X(\theta) = \begin{pmatrix}
    \cos(\theta / 2)  & i\sin(\theta / 2)\\
    i\sin(\theta / 2) & \cos(\theta / 2)
  \end{pmatrix}
  \qquad\text{and}\qquad
  U_Z(\theta) = \begin{pmatrix}
    e^{i\theta/2} & 0\\
    0             & e^{-i\theta/2}
  \end{pmatrix}.
\]
When~$ab \neq 0$, the angles~$\alpha$, $\beta$, and~$\gamma$ can be uniquely
chosen in the intervals~$[0, 2\pi)$, $[0, \pi]$, and~$[-2\pi, 2\pi)$, and are
called the \defi{Euler angles} of~$U$.

Define the map~$\xi\colon\R^3 \to \{\, M \in \C^{2\times 2} : M = M^*\text{ and
}\trace M = 0\,\}$ by
\[
  \xi(x) = \begin{pmatrix}
    -x_3       & x_1 - ix_2\\
    x_1 + ix_2 & x_3
  \end{pmatrix}.
\]
This is an invertible linear map from~$\R^3$ to the real space of~$2 \times 2$
Hermitian traceless matrices.

\begin{theorem}%
  \label{thm:suni-sorto-homomorphism}
  The map~$\Phi\colon \suni(2) \to \R^{3 \times 3}$ such that
  \[
    \Phi(U)x = \xi^{-1}(U\xi(x)U^*)
  \]
  for all~$x \in \R^3$ is a surjective homomorphism from~$\suni(2)$
  to~$\sorto(3)$ with kernel~$\{I, -I\}$.
\end{theorem}

\begin{proof}
A simple calculation shows that~$\Phi(I) = I$ and~$\Phi(U_1 U_2) = \Phi(U_1)
\Phi(U_2)$, and so~$\Phi$ is a group homomorphism. Moreover,~$\Phi(U_X(\theta))
= R_X(\theta)$ and~$\Phi(U_Z(\theta)) = R_Z(\theta)$, and so~$\Phi$ is a
surjective homomorphism from~$\suni(2)$ to~$\sorto(3)$.

It remains to show that the kernel of~$\Phi$ is~$\{I, -I\}$.  Indeed,~$\Phi(I) =
\Phi(-I) = I$.  If~$\Phi(U) = I$ for some~$U \in \suni(2)$, then for all~$x \in
\R^3$ we have~$x = \Phi(U)x = \xi^{-1}(U\xi(x)U^*)$, hence~$U\xi(x) = \xi(x)U$.
So~$U$ has to commute with every Hermitian traceless matrix, implying that~$U =
\pm I$.
\end{proof}

From this theorem, it follows that an irreducible representation of~$\sorto(3)$
gives an irreducible representation~$\rho$ of~$\suni(2)$ such that~$\rho(I) =
\rho(-I)$.  Conversely, an irreducible representation~$\rho$ of~$\suni(2)$
with~$\rho(I) = \rho(-I)$ gives an irreducible representation of~$\sorto(3)$.
Hence our goal becomes describing the irreducible representations of~$\suni(2)$.

For a half-integer~$k \geq 0$, let~$V_k$ be the space of homogeneous polynomials
of degree~$2k$ in~$\C[z_1, z_2]$.  (The use of half-integer indices seems odd at
first, but it is classical and leads to cleaner formulas.)  A matrix~$U \in
\suni(2)$ acts on~$p \in V_k$ by
\[
  (U \cdot p)(z) = p(U^* z),
\]
where~$z = (z_1, z_2)$.  This action gives the representation~$\rho_k$
of~$\suni(2)$ on~$V_k$.

The representations~$\rho_k$ form a complete list of nonequivalent irreducible
representations of~$\suni(2)$ (see Bröcker and tom Dieck~\cite{BrockerD1985}).
For~$p \in V_k$ we have~$-I \cdot p = (-1)^{2k} p$, so given our previous
observation we get from each representation~$\rho_k$ of~$\suni(2)$ for
integer~$k \geq 0$ a representation~$\tau_k$ of~$\sorto(3)$, and these form a
complete list of nonequivalent irreducible representations of~$\sorto(3)$.


\subsection{Concrete formulas}

To use Bochner's theorem effectively, we need to have concrete formulas for the
representations of~$\suni(2)$.  Let~$p \in \C[z_1, z_2]$.  We denote
by~$\overline{p}$ the polynomial obtained from~$p$ by conjugating each
coefficient.  We denote by~$p(\partial)$ the differential operator obtained
from~$p$ by replacing each occurrence of~$z_i^n$ by~$\partial^n/\partial z_i^n$.
For instance, if~$p(z_1, z_2) = 2z_1^2 z_2 - 3z_1^3$, then
\[
  p(\partial) = 2\frac{\partial^3}{\partial z_1^2\partial z_2} -
  3 \frac{\partial^3}{\partial z_1^3}.
\]
We can define an inner product on~$V_k$ by
\[
  \langle p, q\rangle = \overline{p}(\partial) q.
\]
This inner product is invariant under the action of~$\suni(2)$, and so the
representation~$\rho_k$ on~$V_k$ is unitary when~$V_k$ is equipped with this
inner product.

For real numbers~$a < b$, let
\[
  \zint[]{a}{b} = \{\, a + k : k \in \Z,\ k \geq 0,\ a + k \leq b\,\}.
\]
Fix a half-integer~$k \geq 0$.  Our goal is to compute the entries of the matrix
representation of~$\rho_k$.  For~$n \in \zint[]{-k}{k}$, let
\[
  e_n = ((k-n)! (k+n)!)^{-1/2} z_1^{k-n} z_2^{k+n}.
\]
Direct verification shows that the polynomials~$e_n$ form an orthonormal basis
of~$V_k$.  Given~$U \in \suni(2)$ and~$m$, $n \in \zint[]{-k}{k}$, we want to
compute the~$(m, n)$ entry of~$\rho_k(U)$, namely
\[
  \rho_k(U)_{mn} = \langle e_m, U \cdot e_n\rangle.
\]

If~$p \in V_k$, then
\[
  \langle e_m, p\rangle = \biggl(\frac{(k+m)!}{(k-m)!}\biggr)^{1/2}
  \frac{d^{k-m}}{dz^{k-m}} p(z, 1) \biggr|_{z=0}.
\]
(Check it for~$p = e_n$ for~$n \in \zint[]{-k}{k}$.) If~$U$ is given as
in~\eqref{eq:suni-matrix}, then
\[
  U \cdot e_n = ((k-n)! (k+n)!)^{-1/2} (\overline{a} z_1 - b z_2)^{k-n}
  (\overline{b} z_1 + a z_2)^{k+n}
\]
and hence
\begin{equation}%
  \label{eq:rho-entry}
  \rho_k(U)_{mn} = C_{mn} \frac{d^{k-m}}{dz^{k-m}} (\overline{a} z - b)^{k-n}
  (\overline{b} z + a)^{k+n} \biggr|_{z=0},
\end{equation}
where
\[
  C_{mn} = \biggl(\frac{(k+m)!}{(k-m)! (k-n)! (k+n)!}\biggr)^{1/2}.
\]

If we represent~$U$ via Euler angles as~$U = U(\alpha, \beta, \gamma)$, then
\[
  \rho_k(U) = \rho_k(U_Z(\gamma)) \rho_k(U_X(\beta)) \rho_k(U_Z(\alpha)).
\]
Hence it suffices to compute the matrix entries of~$\rho_k(U_X(\beta))$
and~$\rho_k(U_Z(\theta))$.

Set~$a = e^{i\theta/2}$ and~$b = 0$ in~\eqref{eq:rho-entry} to conclude
that~$\rho_k(U_Z(\theta))$ is diagonal and
\begin{equation}%
  \label{eq:UZ-entry}
  \rho_k(U_Z(\theta))_{mm} = e^{im\theta}.
\end{equation}

The case of~$\rho_k(U_X(\beta))$ is more complicated.  For~$w \in [-1, 1]$,
let
\[
  U(w) = \begin{pmatrix}
    ((1+w)/2)^{1/2}  & i((1-w)/2)^{1/2}\\
    i((1-w)/2)^{1/2} & ((1+w)/2)^{1/2}
  \end{pmatrix}.
\]
Note that~$U(w) \in \suni(2)$.  The trigonometric identities
\begin{equation}%
  \label{eq:trig-identities}
  \sin^2(\beta/2) = (1 - \cos\beta)/2\qquad\text{and}\qquad
  \cos^2(\beta/2) = (1 + \cos\beta)/2
\end{equation}
imply that~$U_X(\beta) = U(\cos\beta)$ for~$\beta \in [0, \pi]$.

For a half-integer~$k \geq 0$,~$m$, $n \in \zint[]{-k}{k}$, and~$w \in [-1, 1]$,
set
\[
  P_{mn}^k(w) = i^{m-n} \rho_k(U(w))_{mn}.
\]
The following theorem shows how~$P_{mn}^k$ can be expressed in terms of Jacobi
polynomials; see also~\S9.14 of Andrews, Askey, and Roy~\cite{AndrewsAR1999} or
Chapter~9 of Chirikjian and Kyatkin~\cite{ChirikjianK2001}.

\begin{theorem}
  If~$k \geq 0$ is half-integer, if~$m$, $n \in \zint[]{-k}{k}$, and if~$w \in
  [-1, 1]$, then
  \begin{equation}%
    \label{eq:pmn-relations}
    P_{mn}^k(w) = (-1)^{m-n} P_{nm}^k(w)\quad\text{and}\quad
    P_{mn}^k(w) = (-1)^{m-n} P_{-m, -n}^k(w).
  \end{equation}
  Moreover, if~$m-n$, $m+n \geq 0$, then
  \begin{multline}%
    \label{eq:pmn-w-formula}
    P_{mn}^k(w) = \biggl(\frac{(k-m)!(k+m)!}{(k-n)!(k+n)!}\biggr)^{1/2}
    2^{-m} (1-w)^{(m-n)/2} (1+w)^{(m+n)/2}\\{}\cdot P_{k-m}^{m-n, m+n}(w),
  \end{multline}
  where~$P_k^{(\alpha, \beta)}$ is the Jacobi polynomial of degree~$k$ with
  parameters~$\alpha$, $\beta$ in the standard normalization.
\end{theorem}

\begin{proof}
If~$p \in \C[z]$ is a polynomial, then~$\overline{p}'(z) = \overline{p'}(z)$,
that is, the derivative of the conjugate polynomial is the conjugate of the
derivative.  From~\eqref{eq:rho-entry} we then get
\[
  \rho_k(U(w))_{mn} = \overline{\rho_k(U(w)^*)_{mn}} =
  \overline{\rho_k(U(w))^*_{mn}} = \rho_k(U(w))_{nm}.
\]
This establishes the first identity in~\eqref{eq:pmn-relations}.

For the second identity, use the same observation to see
from~\eqref{eq:rho-entry} that
\[
  \overline{\rho_k(U(-w))_{m,-n}} = (-1)^k \rho_k(U(w))_{mn}
\]
and then to finish apply this identity twice together with the previous one.

To find a formula for~$P_{mn}^k$, let~$w \in (-1, 1)$.  Set~$a =
((1+w)/2)^{1/2}$ and~$b = i((1-w)/2)^{1/2}$; note that~$a$, $b \neq 0$. With~$x
= 2\overline{ab} z - 2 b \overline{b} + 1$ we have
\[
  \overline{a} z - b = -\frac{1 - x}{2\overline{b}}\qquad\text{and}\qquad
  \overline{b} z + a = \frac{1 + x}{2\overline{a}}.
\]
Make this substitution in~\eqref{eq:rho-entry} and apply the chain rule
repeatedly to get
\begin{multline*}
  \rho_k(U(w)) = C_{mn} 2^{-k} (-1)^{k-m} i^{n-m} (1 - w)^{(n-m)/2}
  (1+w)^{-(n+m)/2}\\
  {}\cdot \frac{d^{k-m}}{dx^{k-m}} (1-x)^{k-n} (1+x)^{k+n} \biggr|_{x=w}.
\end{multline*}
Since~$m-n$, $m+n \geq 0$, we can apply Rodrigues's formula ((2.5.13$'$) in
Andrews, Askey, and Roy~\cite{AndrewsAR1999}) to get~\eqref{eq:pmn-w-formula}.
The validity of the formula for~$w = \pm 1$ can be checked directly.
\end{proof}

Formula~\eqref{eq:pmn-w-formula} is simple, but only holds when~$m-n$, $m+n \geq
0$.  Using~\eqref{eq:pmn-relations}, however, we can always reduce to this case.

Finally, since~$U_X(\beta) = U(\cos\beta)$, using~\eqref{eq:trig-identities} we
get the following formula for $\beta \in [0, \pi]$:
\begin{multline*}%
  P_{mn}^k(\cos\beta) = \biggl(\frac{(k-m)!(k+m)!}{(k-n)!(k+n)!}\biggr)^{1/2}
  \sin^{m-n}(\beta/2) \cos^{m+n}(\beta/2)\\
  {}\cdot P_{k-m}^{(m-n, m+n)}(\cos\beta).
\end{multline*}
Together with~\eqref{eq:UZ-entry} this gives the~$(m, n)$ matrix entry for a
general rotation:
\begin{equation}%
  \label{eq:rhok-entries}
  \rho_k(U(\alpha, \beta, \gamma))_{mn} =
  i^{m-n} e^{i(m\gamma + n\alpha)} P_{mn}^k(\cos\beta).
\end{equation}

We now have the complete picture of the representation theory of~$\suni(2)$.
In~\S\ref{sec:sorto3-reps} we have seen how the representation theory
of~$\sorto(3)$ is related to that of~$\suni(2)$.  Indeed, let~$\Phi$ be the
surjective homomorphism of Theorem~\ref{thm:suni-sorto-homomorphism}.  For
integer~$k \geq 0$ and~$U \in \suni(2)$, let
\[
  \tau_k(\Phi(U)) = \rho_k(U).
\]
Since~$k$ is integer we have~$\rho_k(-I) = \rho_k(I)$, so~$\tau_k$ is well
defined and hence a representation of~$\sorto(3)$ on~$V_k$.  We observed
in~\S\ref{sec:sorto3-reps} that the representations~$\tau_k$ for integer~$k \geq
0$ form a complete list of nonequivalent irreducible representations
of~$\sorto(3)$.

Since~$\Phi(U_X(\theta)) = R_X(\theta)$ and~$\Phi(U_Z(\theta)) = R_Z(\theta)$ we
have
\[
  \Phi(U(\alpha, \beta, \gamma)) = R(\alpha, \beta, \gamma),
\]
and so we get a formula for the matrix entries of~$\tau_k$ directly
from~\eqref{eq:rhok-entries}.

For a function~$f$ with domain~$\sorto(3)$, let~$f(\alpha, \beta, \gamma) =
f(R(\alpha, \beta, \gamma))$ for short.  We get the following specialization of
Bochner's theorem (Theorem~\ref{thm:Bochner}):

\begin{theorem}%
  \label{thm:Bochner-specialized}
  A function~$f\colon \sorto(3)^2 \to \C$ is continuous and of positive type if
  and only if there are Hermitian positive semidefinite matrices~$F_k$ for~$k
  \geq 0$ with~$\sum_{k\geq 0} \trace F_k < \infty$ such that
  \begin{equation}%
    \label{eq:sorto3-bochner}
    f(\alpha, \beta, \gamma) = \sum_{k \geq 0} \sum_{m, n = -k}^k
    \overline{(F_k)_{mn}}\, i^{m-n} e^{i(m\gamma + n\alpha)}
    P_{mn}^k(\cos\beta),
  \end{equation}
  in which case the series converges absolutely and uniformly.
\end{theorem}


\section{Computing the theta number via trigonometric sums of squares}%
\label{sec:sos}

In~\S\ref{sec:harmonic} we have seen how to parameterize a function~$f\colon
\sorto(3) \to \C$ of positive type.  In this section we will see how to use this
parametrization together with trigonometric sums of squares to transform the
problem of finding a feasible solution of~\eqref{opt:theta-func} into a
semidefinite programming problem, leading to the results presented
in~\S\ref{sec:results}.

Our focus on packings of regular spherical polygons is essential for two main
reasons.  First, because it allows us to use the concept of admissibility
(Definition~\ref{def:admissible}) to transform our optimization problem into a
trigonometric sum-of-squares problem and then into a semidefinite programming
problem.  Second, because the symmetry of the polygons can be exploited to
reduce the size of the resulting semidefinite program, leading to
computationally tractable problems.

A standard regular $N$-gon (see~\S\ref{sec:notation} for definitions) is
invariant under the action of the subgroup of~$\sorto(3)$ generated by the
rotations
\[
  R_Z(2k\pi/N)\qquad\text{for~$k = 0$, \dots,~$N - 1$},
\]
which is isomorphic to~$\Z_N$; we henceforth identify this subgroup with~$\Z_N$
via the bijection~$k \leftrightarrow R_Z(2k\pi/N)$.  The standard regular
$N$-gon is also invariant under the action of~$\Z_2$ that flips the sign of the
first coordinate, that is, for~$x = (x_1, x_2, x_3)$ the action is
\[
  0 \cdot x = x\qquad\text{and}\qquad 1 \cdot x = (-x_1, x_2, x_3).
\]


\subsection{Invariance under~$\Z_N^2$ and~$\Z_2^3$}%
\label{sec:f-invariance}

The group~$\Z_N^2$ acts on continuous functions on~$\sorto(3)$ by
\[
  ((k, \ell) \cdot f)(S) = f(R_Z(2k\pi/N)^{-1} S R_Z(2\ell\pi/N)).
\]

Let~$K$ be a standard regular $N$-gon.  The invariance of~$K$ under~$\Z_N$
implies that, for~$S \in \sorto(3)$,
\[
  \begin{split}
    &K^\circ \cap R_Z(2k\pi/N)^{-1} S R_Z(2\ell\pi/N)
    K^\circ = \emptyset\\
    &\qquad\iff\quad R_Z(2k\pi/N) K^\circ \cap S R_Z(2\ell\pi/N) K^\circ =
    \emptyset\\
    &\qquad\iff\quad K^\circ \cap S K^\circ = \emptyset.
  \end{split}
\]
It follows that if~$f$ is a feasible solution of~\eqref{opt:theta-func}, then so
is~$(k, \ell) \cdot f$ for all~$(k, \ell) \in \Z_N^2$.  We can then average over
the group to get an invariant solution with the same objective value, namely
\[
  \frac{1}{N^2} \sum_{(k, \ell) \in \Z_N^2} (k, \ell) \cdot f.
\]

So we see that we can restrict ourselves in~\eqref{opt:theta-func}
to~$\Z_N^2$-invariant functions without sacrificing the quality of the bound.
This restriction leads to a reduction on the size of the coefficient matrices we
need to express~$f$ as in~\eqref{eq:sorto3-bochner}.  Indeed, fix integers~$m$
and~$n$ and do the averaging for the exponential term
in~\eqref{eq:sorto3-bochner} to get
\[
  \frac{1}{N^2} \sum_{(k, \ell) \in \Z_N^2} e^{i(m(\gamma - 2k\pi/N) + n(\alpha
  + 2\ell\pi/N))} = e^{i(m\gamma + n\alpha)} \frac{1}{N^2} \sum_{(k, \ell) \in
  \Z_N^2} e^{-2mk\pi i/N} e^{2n\ell\pi i/N}.
\]
This equals~$0$ unless both~$m$ and~$n$ are multiples of~$N$, in which case the
result is~$e^{i(m\gamma + n\alpha)}$.

The conclusion is that we can restrict~$m$ and~$n$ in the inner sum
in~\eqref{eq:sorto3-bochner} to run over multiples of~$N$.  For integer~$k \geq
0$ and Euler angles~$\alpha$, $\beta$, and~$\gamma$, let~$Y_k(\alpha, \beta,
\gamma)$ be the matrix indexed by all multiples of~$N$ in~$[-k, k]$ given by
\[
  Y_k(\alpha, \beta, \gamma)_{mn} = i^{m-n} e^{i(m\gamma + n\alpha)}
  P_{mn}^k(\cos\beta).
\]
With this notation, the expression~\eqref{eq:sorto3-bochner} for a
$\Z_N^2$-invariant function~$f$ becomes
\[
  f(\alpha, \beta, \gamma) = \sum_{k \geq 0} \langle F_k, Y_k(\alpha, \beta,
  \gamma)\rangle.
\]

Euler rotations satisfy the relations
\[
  \begin{split}
    R(\alpha, \beta, \gamma)^{-1} &= R(-\gamma, -\beta, -\alpha) = R(\pi - \gamma,
    \beta, \pi - \alpha),\\
    R(\alpha, \beta, \gamma) &= R(\alpha - \pi, -\beta, \gamma - \pi),
  \end{split}
\]
as can be checked via trigonometric identities.  A function~$f\colon \sorto(3)
\to \R$ of positive type satisfies~$f(S^{-1}) = f(S)$, and so with the
identities above we get the relations
\begin{equation}%
  \label{eq:f-real-relations}
    f(\alpha, \beta, \gamma) = f(\alpha - \pi, -\beta, \gamma - \pi) = f(-\gamma,
    -\beta, -\alpha) = f(\pi - \gamma, \beta, \pi - \alpha).
\end{equation}

Let~$M$ be the matrix of the linear transformation~$(x_1, x_2, x_3) \mapsto
(-x_1, x_2, x_3)$.  We have seen that the standard regular $N$-gon~$K$ is
invariant under~$M$, that is,~$MK = K$.  So, in analogy to the case of~$\Z_N^2$
discussed above, we get, for~$S \in \sorto(3)$,
\[
  K^\circ \cap M S M K^\circ = \emptyset\qquad\iff
  \qquad K^\circ \cap S K^\circ = \emptyset.
\]
This allows us to restrict~\eqref{opt:theta-func} further by requiring~$f$ to be
invariant under the action of~$\Z_2$ given by
\[
  (0 \cdot f)(S) = f(S)\qquad\text{and}\qquad (1 \cdot f)(S) = f(MSM).
\]
Since $MR_X(\beta)M = R_X(\beta)$ and so
\[
  \begin{split}
    MR(\alpha, \beta, \gamma)M &= MR_Z(\gamma) R_X(\beta) R_Z(\alpha)M\\
    &= M R_Z(\gamma) M R_X(\beta) M R_Z(\alpha) M\\
    &= R(-\alpha, \beta, -\gamma),
  \end{split}
\]
this is the same as imposing the relation
\begin{equation}%
  \label{eq:f-z2-relation}
  f(\alpha, \beta, \gamma) = f(-\alpha, \beta, -\gamma).
\end{equation}

Combining~\eqref{eq:f-real-relations} and~\eqref{eq:f-z2-relation} we see
that~$f$ can be assumed to be invariant under the action of~$\Z_2^3$ generated
by
\begin{equation}%
  \label{eq:z3-action}
  \begin{split}
    ((1, 0, 0) \cdot f)(\alpha, \beta, \gamma) &= f(-\alpha, \beta, -\gamma),\\
    ((0, 1, 0) \cdot f)(\alpha, \beta, \gamma) &= f(-\gamma, -\beta, -\alpha),\\
    ((0, 0, 1) \cdot f)(\alpha, \beta, \gamma) &= f(\pi - \gamma, \beta, \pi -
    \alpha).
  \end{split}
\end{equation}

Since~$Y_k(\pi - \gamma, \beta, \pi - \alpha) = Y_k(-\gamma, -\beta, -\alpha) =
Y_k(\alpha, \beta, \gamma)^*$, by setting
\[
  \overline{Y}_k(\alpha, \beta, \gamma) = (1/4)(Y_k(\alpha, \beta, \gamma) +
  Y_k(\alpha, \beta, \gamma)^* + Y_k(-\alpha, \beta, -\gamma) + Y_k(-\alpha,
  \beta, -\gamma)^*)
\]
we get the following version of Theorem~\ref{thm:Bochner-specialized}.

\begin{theorem}%
  \label{thm:Bochner-invariance}
  A function~$f\colon \sorto(3)^2 \to \R$ is continuous, invariant under the
  action of both~$\Z_N^2$ and~$\Z_2^3$, and of positive type if
  and only if there are Hermitian positive semidefinite matrices~$F_k$ for~$k
  \geq 0$ with~$\sum_{k\geq 0} \trace F_k < \infty$ such that
  \[
    f(\alpha, \beta, \gamma) = \sum_{k \geq 0} \langle F_k,
    \overline{Y}_k(\alpha, \beta, \gamma)\rangle,
  \]
  in which case the series converges absolutely and uniformly.
\end{theorem}


\subsection{Trigonometric sums of squares}%
\label{sec:trig-sos}

Given variables~$\theta = (\theta_1, \ldots, \theta_n)$, denote by~$\C\llbracket
\theta_1, \ldots, \theta_n\rrbracket$ the ring of polynomials on~$e^{\pm
i\theta_k}$, that is,
\[
  \C\llbracket\theta_1, \ldots, \theta_n\rrbracket = \C[e^{i\theta_1},
  e^{-i\theta_1}, \ldots, e^{i\theta_n}, e^{-i\theta_n}].
\]
An element~$p$ of~$\C\llbracket \theta_1, \ldots, \theta_n\rrbracket$ is an
expression of the form
\[
  p(\theta_1, \ldots, \theta_n) = \sum_{k \in \Z^n} a_k e^{i k \cdot \theta},
\]
where the coefficients~$a_k$ are complex numbers and only finitely many of them
are nonzero.  If~$a_{-k} = \overline{a_k}$ for all~$k$, then~$p$ is called a
\defi{(Hermitian) trigonometric polynomial}.  This condition on the coefficients
implies that a trigonometric polynomial is real valued. Both~$\cos\theta$
and~$\sin\theta$ are trigonometric polynomials, since
\[
  \cos\theta = (1/2)(e^{i\theta} + e^{-i\theta})\qquad\text{and}\qquad
  \sin\theta = (i/2)(e^{-i\theta} - e^{i\theta}).
\]

The \defi{degree} of~$\theta_i$ in the polynomial~$p$ is
\[
  \deg_i p = \max\{\, |k_i| : k \in \Z^n,\ a_k \neq 0\,\}.
\]
The \defi{degree} of~$p$ is~$\deg p = (\deg_1 p, \ldots, \deg_n p)$.

A trigonometric polynomial~$\sigma \in \C\llbracket \theta_1, \ldots,
\theta_n\rrbracket$ is a \defi{sum of squares (sos)} if there are
polynomials~$q_k \in \C\llbracket \theta_1, \ldots, \theta_n\rrbracket$ for~$k =
1$, \dots,~$N$ such that
\[
  \sigma(\theta) = \sum_{k=1}^N q_k(\theta) \overline{q_k(\theta)}.
\]
Note that a trigonometric sos polynomial is always nonnegative.

Equivalently,~$\sigma$ is a trigonometric sos polynomial if there is a vector~$v
= v(\theta)$ whose entries are polynomials in~$\C\llbracket \theta_1, \ldots,
\theta_n\rrbracket$ and a Hermitian positive-semidefinite matrix~$Q$ such that
\[
  \sigma(\theta) = \langle Q, v(\theta) v(\theta)^*\rangle.
\]
This is a \defi{Gram-matrix representation} of~$\sigma$. For instance, if~$v$
consists of all monomials~$e^{i k \cdot \theta}$ for~$0 \leq k \leq \delta$ and
some~$\delta \in \Z^n$, then
\[
  \sigma(\theta) = \sum_{0 \leq k, \ell \leq \delta} \overline{Q_{k\ell}} e^{i
  k\cdot \theta} e^{-i \ell\cdot \theta}.
\]

Given a domain~$\Delta = \{\, \theta : g_1(\theta) \geq 0, \ldots, g_N(\theta)
\geq 0\,\}$ for some trigonometric polynomials~$g_i$, if the trigonometric
polynomial~$p$ can be written as
\[
  p = \sigma_0 + \sum_{i=1}^N g_i \sigma_i,
\]
where the~$\sigma_i$ are trigonometric sos polynomials, then~$p$ is nonnegative
on~$\Delta$.

A reference on trigonometric sos polynomials is the book by
Dumitrescu~\cite{Dumitrescu2017}.  Because we need to exploit the symmetry of
the resulting optimization problem, the definitions given above are slightly
more general than those in Dumitrescu's book.


\subsection{A sum-of-squares formulation}%
\label{sec:sos-formulation}

Let~$K$ be the standard regular $N$-gon with angular radius~$\varrho$. To
solve~\eqref{opt:theta-func} for~$K$, we use
Theorem~\ref{thm:Bochner-invariance} to specify~$f$ by fixing an integer~$d \geq
0$ and setting
\begin{equation}%
  \label{eq:f-exp}
  f(\alpha, \beta, \gamma) = \sum_{k=0}^d \langle F_k, \overline{Y}_k(\alpha,
  \beta, \gamma)\rangle
\end{equation}
for Hermitian positive-semidefinite matrices~$F_k$ of appropriate size.  The
function~$f$ defined above is in fact a trigonometric polynomial, allowing us to
use a sum-of-squares approach to model the constraints
of~\eqref{opt:theta-func}.  Indeed,~$f$ is a trigonometric polynomial
on~$\alpha$, $\beta$, and~$\gamma$ of degree~$(N\floor{d / N}, d, N\floor{d /
N})$ (recall from~\S\ref{sec:f-invariance} how the symmetry of the $N$-gon being
packed affects the definition of~$f$).

We want to use trigonometric sos to model the constraint
\begin{equation}%
  \label{eq:edge-constraint}
  f(S) \leq -1\quad\text{for~$S \in \sorto(3)$ such that~$K^\circ \cap S K^\circ
  = \emptyset$.}
\end{equation}
To this end we need a characterization of the set
\[
  \Delta = \{\, (\alpha, \beta, \gamma) : K^\circ \cap R(\alpha, \beta, \gamma)
  K^\circ = \emptyset\,\}
\]
in terms of trigonometric polynomials.  Let~$a_i^\tp x \leq 0$, for~$i = 0$,
\dots,~$N - 1$, be the inequalities inducing the facets of~$K$, so that
\[
  K = \{\, x \in \R^3 : a_i^\tp x \leq 0\text{ for~$i = 0$, \dots,~$N - 1$}\,\}.
\]
If~$S \in \sorto(3)$, then the facets of~$SK$ are induced by the
inequalities~$(Sa_i)^\tp x \leq 0$ for~$i = 0$, \dots,~$N - 1$.

To get a representation of~$\Delta$, we assume that~$K$ is admissible. The set
of all Euler angles for which condition~(i) of Definition~\ref{def:admissible}
holds is
\[
\Sigma_1 = \{\, (\alpha, \beta, \gamma) : \cos(2\varrho) - \cos\beta \geq
0\,\}.
\]
Similarly, the sets of Euler angles for which conditions~(ii) and~(iii) of
Definition~\ref{def:admissible} hold are, respectively,
\[
  \begin{split}
    \Sigma_2 &= \bigcup_{i=0}^{N-1}
    \{\, (\alpha, \beta, \gamma) : a_i^\tp R(\alpha, \beta, \gamma) u_j \geq
    0\text{ for~$j = 0$, \dots,~$N - 1$}\,\}\quad\text{and}\\
    \Sigma_3 &= \bigcup_{i=0}^{N-1}
    \{\, (\alpha, \beta, \gamma) : (R(\alpha, \beta, \gamma) a_i)^\tp u_j \geq
    0\text{ for~$j = 0$, \dots,~$N - 1$}\,\},
  \end{split}
\]
where~$u_0$, \dots,~$u_{N-1}$ are the vertices of~$K$.  We then have~$\Delta =
\Sigma_1 \cup \Sigma_2 \cup \Sigma_3$.

This description of~$\Delta$ allows us to use the techniques
of~\S\ref{sec:trig-sos} to model~\eqref{eq:edge-constraint} in terms of
trigonometric sos polynomials.  For instance, the constraint on~$\Sigma_1$ can
be expressed as the trigonometric sos constraint
\begin{equation}%
  \label{eq:sigma1-constraint}
  f(\alpha, \beta, \gamma) + \sigma_0(\alpha, \beta, \gamma) + (\cos(2\varrho) -
  \cos\beta)\sigma_1(\alpha, \beta, \gamma) = -1,
\end{equation}
where~$\sigma_0$ and~$\sigma_1$ are trigonometric sos polynomials.  The
constraint on~$\Sigma_2$ translates into~$N$ identities like above, each
involving~$N + 1$ trigonometric sos polynomials, and similarly for~$\Sigma_3$.
Hence we can generate a semidefinite programming problem whose feasible
solutions give upper bounds for the packing number of~$K$.

The invariance relations satisfied by~$f$ reduce the number of trigonometric sos
constraints we need.  Indeed, we have
\[
  (\alpha, \beta, \gamma) \in \Sigma_3\qquad\iff\qquad
  (-\gamma, -\beta, -\alpha) \in \Sigma_2,
\]
and since~$f$ satisfies~\eqref{eq:f-real-relations} it suffices to
ensure~\eqref{eq:edge-constraint} on~$\Sigma_1$ and~$\Sigma_2$.

The constraint on~$\Sigma_2$ can also be simplified.  Indeed, let
\[
  \Sigma_2' = \{\, (\alpha, \beta, \gamma) : a_0^\tp R(\alpha, \beta, \gamma)
  u_j \geq 0\text{ for~$j = 0$, \dots,~$N - 1$}\,\}.
\]
The invariance of~$K$ under~$\Z_N$ implies that the action of~$\Z_N$
on~$\Sigma_2'$ gives the set~$\Sigma_2$, and so since~$f$ is $\Z_N^2$-invariant,
it suffices to ensure~\eqref{eq:edge-constraint} on~$\Sigma_1$ and~$\Sigma_2'$,
reducing the number of trigonometric sos constraints to~$2$.

This greatly reduces the size of the resulting semidefinite program, but still
not enough to make it tractable.  The problem here is that the degrees of the
trigonometric sos polynomials are still too high, and that the matrices needed
to specify them, which are the variables of our optimization problem, are
complex.  Semidefinite programming solvers work with real matrices, and the
reduction of complex to real variable matrices doubles the size of each matrix.

This leads us to our next goal, namely exploiting the problem's invariance in
the construction of the trigonometric sos constraints.


\subsection{Invariant trigonometric sos: warm up}%
\label{sec:warmup}

The theory of invariant semidefinite programs, and in particular of invariant
sums of squares, is by now well developed~\cite{BachocGSV2012, GatermannP2004}.
Invariant trigonometric sos work similarly, but with some notable differences.
In this section, we will consider the case of univariate polynomials:
understanding this simpler situation will make the developments of the next
section, where we tackle our original problem, much more natural.
In~\S\ref{sec:sos-compare} we will explore the differences between
trigonometric and classical invariant sos.

Let~$p \in \C\llbracket\theta\rrbracket$ be a trigonometric polynomial on the
variable~$\theta$.  Below,~$z = e^{i\theta}$ is sometimes used for brevity.  The
group~$\Z_N$ acts on~$p$ by
\[
  (k \cdot p)(\theta) = p(\theta + 2k\pi/N).
\]
Let~$v(\theta) = (e^{ir\theta} : r = -d, \ldots, d)$ and say~$p$ is a
trigonometric sos with the Gram-matrix representation~$\langle Q, v(\theta)
v(\theta)^*\rangle$.  If~$p$ is $\Z_N$-invariant, then
\[
  \begin{split}
    p(\theta) &= \frac{1}{N}\sum_{k \in \Z_N} (k\cdot p)(\theta)\\
    &= \sum_{r,s=-d}^d \overline{Q_{rs}} \frac{1}{N}\sum_{k \in \Z_N}
    e^{i(r-s)(\theta + 2k\pi/N)}\\
    &= \sum_{r,s=-d}^d \overline{Q_{rs}} e^{i(r-s)\theta} \frac{1}{N}\sum_{k \in
    \Z_N} e^{i(r-s)2k\pi/N}.
  \end{split}
\]
The inner sum above equals~$1$ if~$r\equiv s \pmod{N}$ and~$0$ otherwise.  The
conclusion is that the matrix~$Q$ can be split into~$N$ diagonal blocks~$Q_k$,
for~$k \in \Z_N$, with rows and columns indexed by all~$r$ such that~$r \equiv k
\pmod{N}$.

As an example, say~$N = 3$ and~$d = 6$.  The blocks of the
matrix~$v(\theta)v(\theta)^*$ corresponding to~$k = 0$ and~$1$ are
\[
  \begin{blockarray}{cccccc}
    & z^6 & z^3 & 1 & z^{-3} & z^{-6}\\
    \begin{block}{c[ccccc]}
      z^{-6} & 1 & z^{-3} & z^{-6} & z^{-9} & z^{-12}\\
      z^{-3} & z^3 & 1 & z^{-3} & z^{-6} & z^{-9}\\
      1      & z^6 & z^3 & 1 & z^{-3} & z^{-6}\\
      z^3    & z^9 & z^6 & z^3 & 1 & z^{-3}\\
      z^6    & z^{12} & z^9 & z^6 & z^3 & 1\\
    \end{block}
  \end{blockarray}
  \qquad\text{and}\qquad
  \begin{blockarray}{ccccc}
    & z^5 & z^2 & z^{-1} & z^{-4}\\
    \begin{block}{c[cccc]}
      z^{-5} & 1 & z^{-3} & z^{-6} & z^{-9}\\
      z^{-2} & z^3 & 1 & z^{-3} & z^{-6}\\
      z      & z^6 & z^3 & 1 & z^{-3}\\
      z^4    & z^9 & z^6 & z^3 & 1\\
    \end{block}
  \end{blockarray}\,.
\]
Note that the second block is a principal submatrix of the first one.  The same
happens with all further blocks.  Hence, instead of using~$N$ blocks, we need
only a single block: we have~$p(\theta) = \langle Q,
v(\theta)v(\theta)^*\rangle$, where
\[
  v(\theta) = (e^{ir\theta} : r = -d, \ldots, d,\ r \equiv 0\pmod{N}).
\]
In other words, working with $\Z_N$-invariant polynomials amounts to working in
the ring~$\C\llbracket N\theta\rrbracket$. This reduces the size of the
matrix~$Q$ by a factor of~$N$.

The group~$\Z_2$ acts on~$p$ by
\[
  (0 \cdot p)(\theta) = p(\theta)\qquad\text{and}\qquad
  (1 \cdot p)(\theta) = p(-\theta).
\]
This action, flipping the sign of the variable, is different from the additive
action we considered before.  To get a block decomposition here, we adapt the
approach of Gatermann and Parrilo~\cite{GatermannP2004}, using a projection
operator~\cite[\S2.6, Theorem~8]{Serre1977} to get a symmetry-adapted basis.

Say~$G$ is an Abelian group acting on~$\C\llbracket\theta\rrbracket$.  The
projection operator associated with an irreducible character~$\chi$ of~$G$ is
the operator~$\Pi_\chi\colon \C\llbracket\theta\rrbracket \to
\C\llbracket\theta\rrbracket$ such that
\[
  \Pi_\chi p = \frac{1}{|G|} \sum_{g \in G} \overline{\chi(g)} (g \cdot p).
\]
The group~$\Z_2$ is Abelian and its irreducible characters are
\[
  \chi_0(g) = 1\qquad\text{and}\qquad
  \chi_1(g) = (-1)^g
\]
for~$g \in \Z_2$.  We get the projection operators
\[
  (\Pi_0 p)(\theta) = (1/2)(p(\theta) +
  p(-\theta))\qquad\text{and}\qquad
  (\Pi_1 p)(\theta) = (1/2)(p(\theta) - p(-\theta)).
\]

Let~$V$ be the subspace generated by the polynomials~$z^{-d}$, \dots,~$z^d$.
Applying the operators above to these polynomials decomposes~$V$ as a direct sum
of the spaces~$V_0$ and~$V_1$ with bases
\[
  \begin{array}{ll}
    z^k + z^{-k}&\text{for~$k = 0$, \dots,~$d$ and}\\
    z^k - z^{-k}&\text{for~$k = 1$, \dots,~$d$,}
  \end{array}
\]
respectively.  Together, these bases form the symmetry-adapted basis for the
action.  Let~$v(\theta)$ be the vector obtained by concatenating these two
bases.

Say~$p$ has the Gram-matrix representation~$p(\theta) = \langle Q, v(\theta)
v(\theta)^*\rangle$.  If~$p$ is $\Z_2$-invariant, then
\[
  p(\theta) = (1/2)(p(\theta) + p(-\theta)) = \langle Q,
  (1/2)(v(\theta)v(\theta)^* + v(-\theta)v(-\theta)^*)\rangle.
\]
The matrix $v(\theta)v(\theta)^* + v(-\theta)v(-\theta)^*$ is block-diagonal:
only entries with row and column both corresponding to basis elements of~$V_0$
or~$V_1$ are nonzero.  So instead of having a single $(2d+1) \times (2d+1)$
block, we have a representation for~$p$ with two smaller blocks of sizes~$(d+1)
\times (d+1)$ and~$d \times d$.  Namely we let
\[
  \begin{split}
    v_0(\theta) &= (z^k + z^{-k} : k = 0, \ldots, d),\\
    v_1(\theta) &= (z^k - z^{-k} : k = 1, \ldots, d)
  \end{split}
\]
to get
\[
  p(\theta) = \langle Q_0, v_0(\theta) v_0(\theta)^*\rangle + \langle
  Q_1, v_1(\theta) v_1(\theta)^*\rangle.
\]

A further simplification is possible.  The matrices~$v_k(\theta)v_k(\theta)^*$
above are real-valued for every~$\theta$, implying that the matrices~$Q_k$ can
be taken to be real as well.  This further reduces the complexity of the
problem, since in semidefinite programming solvers Hermitian
positive-semidefinite matrices are represented by real positive-semidefinite
matrices of twice the size.

It is now simple to combine both approaches to reduce the size of the
Gram-matrix representation of a polynomial that is invariant under both~$\Z_N$
and~$\Z_2$.  Indeed, $\Z_N$-invariance amounts to working in~$\C\llbracket
N\theta\rrbracket$ and all we need to do is apply the projection operators
for~$\Z_2$ to the polynomials~$z^k$ for all~$k = -d$, \dots,~$d$ that are
multiples of~$N$.


\subsection{Invariant trigonometric sos: setup}

The set~$\Sigma_1$ is given by the trigonometric polynomial~$\cos(2\varrho) -
\cos\beta$.  This polynomial is invariant under the actions of~$\Z_N^2$
and~$\Z_2^3$ described in~\S\ref{sec:f-invariance}.  This means that
in~\eqref{eq:sigma1-constraint} we can take~$\sigma_0$ and~$\sigma_1$ to be
invariant as well.

As in~\S\ref{sec:warmup}, the invariance under~$\Z_N^2$ is the simpler one to
enforce by working in the ring~$\C\llbracket N\alpha, \beta, N\gamma\rrbracket$.
The invariance under~$\Z_2^3$ is exploited by computing a symmetry-adapted basis
as in~\S\ref{sec:warmup}.  The group~$\Z_2^3$ is Abelian; its irreducible
characters are in bijection with~$\Z_2^3$: the character that corresponds to~$z
\in \Z_2^3$ is
\[
  \chi_z(g) = (-1)^{z_1 g_1 + z_2 g_2 + z_3 g_3}
\]
for~$g \in \Z_2^3$. The projection operator associated with this character is
\[
  \Pi_z p = \frac{1}{|\Z_2^3|} \sum_{g \in \Z_2^3} \overline{\chi_z(g)} (g \cdot
  p),
\]
where the action on~$p$ is given by~\eqref{eq:z3-action}.

Assume that both~$d$ and~$\floor{d / N}$ are even.  To represent a trigonometric
sos polynomial~$\sigma$ invariant under both~$\Z_N^2$ and~$\Z_2^3$ of
degree~$\delta = (N\floor{d / N}, d, N\floor{d / N})$, for each~$z \in \Z_2^3$
we apply the projection operator~$\Pi_z$ to all monomials
\[
  e^{i (k_1 N \alpha + k_2 \beta + k_3 N \gamma)}
\]
for~$k_1$, $k_3 = -\floor{d / N} / 2, \dots, \floor{d / N} / 2$ and~$k_2 = -d /
2$, \dots,~$d/2$ and compute a basis of this projection.  The elements of this
basis are the components of the vector~$v_\delta^z$.  The invariant
trigonometric sos polynomial~$\sigma$ can then be written as
\[
  \sigma(\alpha, \beta, \gamma) = \sum_{z \in \Z_2^3} \langle Q_z,
  v_\delta^z(\alpha, \beta, \gamma) v_\delta^z(\alpha, \beta, \gamma)^*\rangle
\]
for Hermitian positive-semidefinite matrices~$Q_z$ of appropriate sizes.  In
this way, we reduce from a single large block used to represent~$\sigma$ to~$8$
smaller blocks, one for each irreducible character of~$\Z_2^3$.  Moreover, the
invariance relation~\eqref{eq:f-z2-relation} implies that we can restrict the
matrices~$Q_z$ to be real without loss of generality.  This further reduces the
size of the matrix variables by half.

The domain~$\Sigma_2'$ is given by the trigonometric polynomials
\begin{equation}%
  \label{eq:sigma2-polys}
  a_0^\tp R(\alpha, \beta, \gamma) u_j \geq 0\qquad
  \text{for~$j = 0$, \dots,~$N - 1$}.
\end{equation}
The invariance of~$K$ implies that this domain is invariant under the action
of~$\Z_N$ on~$\alpha$, that is,
\[
  (\alpha, \beta, \gamma) \in \Sigma_2'\qquad\iff\qquad
  (\alpha + 2k\pi / N, \beta, \gamma) \in \Sigma_2'
\]
for all~$k = 0$, \dots,~$N - 1$.

Our specific choice of standard regular $N$-gon also implies that the domain is
invariant under the action of~$\Z_2^2$ generated by
\[
  \begin{split}
    (1, 0) \cdot (\alpha, \beta, \gamma) &= (-\alpha, \beta, -\gamma),\\
    (0, 1) \cdot (\alpha, \beta, \gamma) &= (\alpha - \pi, -\beta, \gamma -
    \pi).
  \end{split}
\]
Indeed, we know that~$R(\alpha, \beta, \gamma) = R(\alpha - \pi, -\beta, \gamma
- \pi)$.  Moreover, our choice of~$K$ implies that~$(a_0)_1 = 0$, and so~$Ma_0 =
a_0$, where~$M$ is the linear transformation~$(x_1, x_2, x_3) \mapsto (-x_1,
x_2, x_3)$.

The domain~$\Sigma_2'$ is itself invariant under the actions of~$\Z_N$
and~$\Z_2^2$ described above, but the polynomials in~\eqref{eq:sigma2-polys} are
not necessarily so.  Indeed, these polynomials are invariant under the action
of~$(0, 1) \in \Z_2^2$, but not under the action of~$(1, 0)$ or under the action
of~$\Z_N$.

We can replace these polynomials by invariant ones, however, by using the
following trick similar to an observation of Leijenhorst and de
Laat~\cite{LeijenhorstL2024}.  Let~$G$ be the subgroup of~$\Z_2^2$ generated
by~$(1, 0)$.  Note that the actions of~$G$ and~$\Z_N$ map polynomials
in~\eqref{eq:sigma2-polys} to each other; all these polynomials belong to the
same orbit under the action of~$\Z_N$.  For~$k = 1$, \dots,~$N$, consider the
polynomial
\[
  s_k(\alpha, \beta, \gamma) = \sum_{\substack{S \subseteq \{0, \ldots, N-1\}\\|S| =
  k}} \prod_{j \in S} a_0^\tp R(\alpha, \beta, \gamma) u_j.
\]
Note that~$s_k$ is invariant under~$\Z_2^2$ and~$\Z_N$.  Moreover,
\[
  \Sigma_2' = \{\,(\alpha, \beta, \gamma) : s_k(\alpha, \beta, \gamma) \geq
  0\text{ for~$k = 1$, \dots,~$N$}\,\}.
\]

We can now reduce the size of the Gram-matrix representation of the
trigonometric sos polynomials needed.  The action of~$\Z_N$ translates to
working in the ring~$\C\llbracket N\alpha, \beta, \gamma\rrbracket$.  The action
of~$\Z_2^2$ is dealt with as the action of~$\Z_2^3$ before.

Namely, the irreducible characters of~$\Z_2^2$ are in bijection with~$\Z_2^2$,
and for each~$z \in \Z_2^2$ we get a projection operator like~$\Pi_z$. To
represent a trigonometric sos polynomial~$\sigma$ invariant under both~$\Z_N$
and~$\Z_2^2$ of degree~$\delta = (N\floor{d/N}, d, N\floor{d/N})$, for each~$z
\in \Z_2^2$ we apply the projection operator to all monomials
\[
  e^{i(k_1 N\alpha + k2\beta + k_3\gamma)}
\]
for~$k_1 = -\floor{d/N} / 2$, \dots,~$\floor{d/N}/2$, $k_2 = -d/2$, \dots,~$d /
2$, and~$k_3 = -N\floor{d/N}/2$, \dots,~$N\floor{d/N}/2$.  A basis of the
resulting projection gives us a new vector of polynomials~$v_\delta^z$ and the
Gram-matrix representation of~$\sigma$ using~4 smaller blocks, namely
\[
  \sigma(\alpha, \beta, \gamma) = \sum_{z \in \Z_2^2} \langle Q_z,
  v_\delta^z(\alpha, \beta, \gamma) v_\delta^z(\alpha, \beta, \gamma)^*\rangle.
\]
Again, we can assume without loss of generality that the matrices~$Q_z$ are real
instead of complex.


\subsection{Invariance in classical vs.\ trigonometric sos polynomials}%
\label{sec:sos-compare}

Let~$p \in \C\llbracket\theta\rrbracket$ be a trigonometric sos polynomial
on~$\theta$.  In~\S\ref{sec:warmup} we have seen how a Gram-matrix
representation for~$p$ is simplified when~$p$ is $\Z_N$-invariant.  Namely, the
$\Z_N$-invariance amounts to working in~$\C\llbracket N\theta\rrbracket$, and
this leads to a reduction of the matrix size by a factor of~$N$.

With classical polynomials, the situation is different.  If an sos
polynomial~$p$ is invariant under the action of a group~$G$, then in general its
Gram-matrix representation decomposes into one block for each irreducible
representation of~$G$.  In contrast to the case of the $\Z_N$-invariance
of~\S\ref{sec:warmup}, the blocks are different and cannot be combined into a
single one.  With trigonometric polynomials we could simply work in the
invariant ring~$\C\llbracket N\theta\rrbracket$.  With classical polynomials,
working in the invariant ring incurs in loss of generality.

This leads to the next issue, namely simultaneous invariance under the
additive~$\Z_N$ action and the~$\Z_2$ action that flips the sign of the
variable. Say that the trigonometric sos polynomial~$p$ is invariant under both
actions. We simplified the Gram-matrix representation by first restricting to
the invariant ring~$\C\llbracket N\theta\rrbracket$ and then computing a
symmetry-adapted basis for the~$\Z_2$ action.  Is it better to consider both
actions separately? And why can we do so?

The latter question is the easier to answer.  Invariance under~$\Z_N$
amounts to restriction to the invariant ring.  After this restriction, all
polynomials computed to deal with the action of~$\Z_2$ are themselves
$\Z_N$-invariant, and we do not ruin $\Z_N$-invariance or lose generality.

As for the former question, let us fix~$N = 4$ for simplicity. Let~$r$ be the
generator of~$\Z_4$ and~$s$ be the generator of~$\Z_2$, so~$r^4 = e$ and~$s^2 =
e$, where~$e$ is the neutral element (we use multiplicative notation for the
groups).  Analyze the action of both groups on a polynomial~$p$ to conclude
that~$srs = r^{-1}$.  So~$p$ is invariant under the group~$\langle r, s : r^4 =
e,\ s^2 = e,\ srs = r^{-1}\rangle$, which is the dihedral group with~$8$
elements~$D_8$.

This group is not Abelian; it has five irreducible representations: four of
degree~$1$ and one of degree~$2$.  If we use the Gatermann and
Parrilo~\cite{GatermannP2004} approach to block-diagonalize the Gram-matrix
representation, we end up with five blocks.  For~$d = 100$, these blocks have
sizes~26, 25, 25, 25, and~50.  If we first restrict to the invariant ring and
then deal with the~$\Z_2$ action, we get two blocks of sizes~26 and~25, a much
better decomposition.

With classical polynomials we cannot restrict to the invariant ring, but why can
we not deal with actions separately as above?  Let us consider a very similar
situation.  Let~$G$ be the matrix group generated by~$R =
\smallpmatrix{0&1\\-1&0}$ and let~$H$ be the matrix group generated by~$S =
\smallpmatrix{-1&0\\0&1}$; note that~$G$ is isomorphic to~$\Z_4$ and~$H$ is
isomorphic to~$\Z_2$.

A $2 \times 2$ nonsingular matrix~$A$ acts on a two-variable polynomial~$p \in
\C[x, y]$ by $(A \cdot p)(x, y) = p(A^{-1} (x, y))$.  Let~$p \in \C[x, y]$ be a
Hermitian sos polynomial invariant under both~$G$ and~$H$.  So the
polynomial~$p$ is invariant under the group generated by both~$R$ and~$S$, which
is isomorphic to~$D_8$.  The Gatermann and Parrilo approach would give us a
decomposition into five blocks, one for each irreducible representation
of~$D_8$.

\enlargethispage{\baselineskip}  

We could try to consider the groups separately.  First, we consider the action
of~$G$, which is isomorphic to~$\Z_4$.  The irreducible character~$\chi_k(g) =
e^{ik\pi/2}$ for~$k = 0$, \dots,~$3$ gives us a vector~$v_k(x, y)$ of
polynomials, leading to the representation
\[
  p(x, y) = \sum_{k=0}^3 \langle Q_k, v_k(x, y) v_k(x, y)^*\rangle.
\]
Now we can apply the projection operators for the~$H$ action to the
polynomials in each~$v_k$, further breaking down each of the four blocks into
two blocks.  This works for the block~$v_0 v_0^*$, since this block corresponds
to the trivial representation and the polynomials in~$v_0$ are $G$-invariant.
This is no longer true for the other blocks: even though the polynomials in~$v_1
v_1^*$, say, are $G$-invariant, the polynomials in~$v_1$ itself are not.  If we
then decompose according to the~$H$ action, we get matrix elements that are
no longer $G$-invariant.

The conclusion is that what makes it possible to consider the action in steps
for trigonometric polynomials is that $\Z_N$-invariance amounts to working in
the invariant ring.  This is not true for classical polynomials, and so the same
approach fails, requiring us to consider the whole group at once and leading to
a worse reduction in matrix size.


\section{The theta number and the volume bound}%
\label{sec:volume-bound}

Theorem~\ref{thm:theta-bound} states that the theta prime number of a proper
cone gives an upper bound on the packing number that is at least as good as the
volume bound.  In this section, we will investigate the relation between the
theta number and the volume bound in a more general setting.

Any independent set problem can be transformed into a packing problem, and for
packing problems the volume bound can be naturally defined.  This leads to an
upper bound for the independence number, and it turns out that the theta number
is always at least as good as this upper bound.  For simplicity, we investigate
the situation for finite graphs, though results extend to infinite packing
graphs like~$G(K)$ as well.

Let~$G = (V, E)$ be a finite graph.  The Lovász theta number
of~$G$ is given as the optimal value of the following semidefinite
programming problem:
\[
  \begin{optprob}
    \vartheta(G) = \text{maximize}&\langle J, A\rangle\\
    &\trace A = 1,\\
    &A(u, v) = 0\quad\text{for all~$uv \in E$,}\\
    &A \in \R^{V\times V}\text{ is positive semidefinite.}
  \end{optprob}
\]
The theta prime number, denoted by~$\vartheta'(G)$, is the optimal value of the
problem above under the extra constraint that~$A$ has to be nonnegative.

A set~$C \subseteq V$ is a \defi{clique} if any two distinct vertices in~$C$ are
adjacent.  The clique number of~$G$, denoted by~$\omega(G)$, is the maximum
cardinality of a clique of~$G$.  An independent set of~$G$ can contain at most
one vertex from each clique of~$G$.  This leads to the following parameter,
which is an upper bound for~$\alpha(G)$:
\begin{equation}%
  \label{opt:clique}
  \begin{optprob}
    \kappa(G) = \text{maximize}&\one^\tp a\\
    &\sum_{v \in C} a(v) \leq 1\quad\text{for every clique~$C \subseteq
    V$,}\\[1ex]
    &a \in \R^V,\ a \geq 0.
  \end{optprob}
\end{equation}
The inequalities above corresponding to cliques of~$G$ are called \defi{clique
inequalities}.  We have~\cite[Chapter~67]{Schrijver2003B}:
\[
  \alpha(G) \leq \vartheta'(G) \leq \vartheta(G) \leq \kappa(G).
\]

Let~$(\Xcal, \mu)$ be a measure space for a finite~$\Xcal$.  Associate to
each~$v \in V$ a set~$S_v \subseteq \Xcal$ of positive measure in such a way
that
\begin{equation}%
  \label{eq:edge-condition}
  uv \in E\qquad\iff\qquad S_u \cap S_v \neq \emptyset.
\end{equation}
Let~$\Scal = \{\, S_v : v \in V\,\}$.

An \defi{$\Scal$-packing} is a collection of pairwise-disjoint sets in~$\Scal$.
The \defi{packing number} of~$\Scal$ is the maximum cardinality of an
$\Scal$-packing.  The independence number of~$G$ is the packing number
of~$\Scal$, which is bounded from above by
\begin{equation}%
  \label{eq:gen-volume-bound}
  \frac{\mu(\Xcal)}{\min\{\, \mu(S_v) : v \in V\,\}}.
\end{equation}
This upper bound for the packing number of~$\Scal$, and hence for the
independence number of~$G$, is the \defi{volume bound}.

This translation of the independent set problem to a packing problem works for
any measure space~$(\Xcal, \mu)$ and sets~$S_v$, as long
as~\eqref{eq:edge-condition} is satisfied.  A possible choice for a graph
without isolated vertices is to take~$\Xcal = E$ with the counting measure and
let~$S_v$ be the set of edges incident on~$v \in V$.  The volume bound then
implies that~$\alpha(G) \leq |E| / \delta(G)$, where~$\delta(G)$ is the minimum
degree of any vertex of~$G$.

The optimal value of the following linear programming problem gives an upper
bound for the packing number of~$\Scal$:
\begin{equation}%
  \label{opt:packing}
  \begin{optprob}
    p^* = \text{maximize}&\sum_{v \in V} a(v)\\[1ex]
    &\sum_{v \in V} a(v) \one_{S_v} \leq \one,\\[1ex]
    &a \in \R^V,\ a \geq 0.
  \end{optprob}
\end{equation}
The dual of this problem is
\[
  \begin{optprob}
    \text{minimize}&\one^\tp z\\
    &\one_{S_v}^\tp z \geq 1\quad\text{for all~$v \in V$,}\\[1ex]
    &z \in \R^\Xcal,\ z \geq 0.
  \end{optprob}
\]
A feasible solution of the dual problem can be obtained by setting
\[
  z(x) = \frac{\mu(\{x\})}{\min\{\, \mu(S_v) : v \in V\,\}}.
\]
It follows that~$p^*$ is at most the volume bound~\eqref{eq:gen-volume-bound}.

Now~\eqref{opt:packing} has one constraint for each element of~$\Xcal$, the
constraint for~$x \in \Xcal$ being
\[
  \sum_{v \in V,\ x \in S_v} a(v) \leq 1.
\]
Note that~$C = \{\, v \in V : x \in S_v\}$ is a clique of~$G$, and so the
constraint above is a clique inequality;  it follows that~$\kappa(G) \leq p^*$.
The conclusion is that~$\kappa$, $\vartheta$, and~$\vartheta'$ are always at
least as good as the bound of~\eqref{opt:packing} and the corresponding volume
bound for \emph{any} packing problem associated with~$G$ as above.

At least for vertex-transitive graphs it is possible to determine the best
possible volume bound from a packing problem associated with~$G$. Indeed, if~$G$
is vertex transitive, then by an averaging argument~\eqref{opt:clique} has an
optimal solution~$a$ that is a multiple of~$\one$, namely~$a = \omega(G)^{-1}
\one$. So~$\kappa(G) = |V| / \omega(G)$, and hence this is a lower bound on any
volume bound for~$G$.  We claim that there is a packing problem associated
with~$G$ whose volume bound is~$\kappa(G)$.

Indeed, let~$\Xcal$ be the set of all cliques of~$G$.  For~$v \in V$, let~$S_v$
be the set of all cliques of~$G$ that contain~$v$.  Note that the sets~$S_v$
satisfy~\eqref{eq:edge-condition}.  Let~$\Gamma$ be the automorphism group
of~$G$ and fix a maximum clique~$C$ of~$G$.  Consider the measure~$\mu$
on~$\Xcal$ such that
\[
  \mu(\{S\}) = \begin{cases}
    1&\text{if~$S = \sigma C$ for some~$\sigma \in \Gamma$;}\\
    0&\text{otherwise.}
  \end{cases}
\]
On the one hand,
\[
  \mu(\Xcal) = |\{\,\sigma C : \sigma \in \Gamma\,\}|
  = \frac{|\Gamma|}{|\{\,\sigma \in \Gamma : \sigma C = C\,\}|}.
\]
On the other hand, for all~$v \in V$ we have
\[
  \mu(S_v) = |\{\, \sigma C : \sigma \in \Gamma,\ v \in \sigma C\,\}| =
  \frac{|\{\,\sigma \in \Gamma : v \in \sigma C\,\}|}{|\{\,\sigma \in \Gamma :
  \sigma C = C\,\}|};
\]
in particular,~$\mu(S_v)$ does not depend on~$v$.  So the volume bound is
\[
  \frac{\mu(\Xcal)}{\mu(S_v)} = \frac{|\Gamma|}{|\{\,\sigma \in \Gamma : v \in
  \sigma C\,\}|} = \frac{|V|}{|C|} = \frac{|V|}{\omega(G)} = \kappa(G),
\]
as claimed.

What about the volume bound of~\S\ref{sec:packing-volume} for a proper cone~$K
\subseteq \R^n$?  Is it the best possible volume bound, equal to the analogue of
the~$\kappa$ parameter for the infinite graph~$G(K)$?

Fix~$e \in S^{n-1}$ and let~$C = \{\, S \in \sorto(n) : e \in SK^\circ\,\}$.
If~$S$, $T \in C$, then~$S K^\circ \cap T K^\circ \neq \emptyset$.  It follows
that~$C$ is a clique of~$G(K)$, hence the clique number of~$G(K)$ is at least
\[
  \mu(C) = \frac{\omega(K \cap S^{n-1})}{\omega(S^{n-1})},
\]
where~$\mu$ is the Haar probability measure on~$\sorto(n)$ and~$\omega$ is the
surface measure on~$S^{n-1}$.  So the volume bound for~$K$ is best possible if
and only if the clique~$C$ defined above is a maximum-measure clique, a fact
that is unclear and likely hard to settle.


\section{Results}%
\label{sec:results}

Results for packings of spherical polygons are given in Table~\ref{tab:bounds},
but to put them into perspective we first discuss bounds for spherical cap
packings.

\begin{table}[t]
  \footnotesize
  \def\ct#1{\multicolumn{1}{c}{#1}}
  \begin{center}
    \begin{tabular}{rrrrrr}
      \ct{$n$\rule[-0.9ex]{0pt}{0pt}} & \ct{\sl Manifold opt.} & \ct{\sl Best radius} & \ct{\sl
      Difference} & \ct{\sl Volume bound} & \ct{$\vartheta'$}\\
      \hline
      5\rlap{${}^\star$}\rule{0pt}{2.6ex}&0.785398163346&0.785398163397&$5.07 \times 10^{-11}$&  6.828427125582 & 6.000000000608\\
      6\rlap{${}^\star$}&0.785398163347&0.785398163397&$5.02 \times 10^{-11}$&                   6.828427125573 & 6.000000000602\\
      7\rlap{${}^\star$}&0.679539948739&0.679539948816&$7.63 \times 10^{-11}$&                   9.003408786469 & 7.688838860349\\
      8\rlap{${}^\star$}&0.653263580805&0.653263580858&$5.37 \times 10^{-11}$&                   9.713664851489 & 8.279523460373\\
      9\rlap{${}^\star$}&0.615479708605&0.615479708670&$6.51 \times 10^{-11}$&                  10.898979487799 & 9.519184645515\\
      10\rlap{${}^\star$}&0.577239916649&0.577239916709&$5.96 \times 10^{-11}$&                 12.343549645544 & 10.791403558300\\
      11\rlap{${}^\star$}&0.553574358834&0.553574358897&$6.26 \times 10^{-11}$&                 13.391435054800 & 12.000000002063\\
      12\rlap{${}^\star$}&0.553574358833&0.553574358897&$6.32 \times 10^{-11}$&                 13.391435054828 & 12.000000002083\\
      13\rlap{${}^\star$}&0.498611796104&0.498611796219&$1.14 \times 10^{-10}$&                 16.426734561716 & 14.496932163957\\
      14\rlap{${}^\star$}&0.485817371369&0.485817371443&$7.37 \times 10^{-11}$&                 17.285126813258 & 15.228084228233\\
      15&0.468253077068&0.468253077154&$8.61 \times 10^{-11}$&                                  18.580122775901 & 16.448613310944\\
      16&0.455918360677&0.455918360799&$1.23 \times 10^{-10}$&                                  19.580404566686 & 17.368859970312\\
      17&0.445847224435&0.445847224575&$1.40 \times 10^{-10}$&                                  20.459446969949 & 18.135487708677\\
      18&0.432463395488&0.432463395991&$5.03 \times 10^{-10}$&                                  21.724035798875 & 19.392814886060\\
      19&0.416190457510&0.416190463888&$6.38 \times 10^{-09}$&                                  23.428995553064 & 20.979270104826\\
      20&0.413913873373&0.413913874894&$1.52 \times 10^{-09}$&                                  23.683688975153 & 21.178648745611\\
      21&0.398047141625&0.398050462820&$3.32 \times 10^{-06}$&                                  25.581897377143 & 22.789038918072\\
      22&0.390431558872&0.390431560119&$1.25 \times 10^{-09}$&                                  26.576271924753 & 23.697057479768\\
      23&0.381441386802&0.381441395654&$8.85 \times 10^{-09}$&                                  27.827647704305 & 24.812525569503\\
      24\rlap{${}^\star$}&0.381273868311&0.381273869375&$1.06 \times 10^{-09}$&                 27.851808806866 & 24.835140479898\\
    \end{tabular}
  \end{center}

  \caption{The table shows, for each number~$n$ of caps, the radius found by the
  manifold optimization method and the best known radius for a packing of~$n$
  caps, truncated to~12 decimal places, together with the difference between the
  two. The theta prime (linear programming) upper bound and the volume bound for
  a packing of caps with the manifold optimization radius are also given; the
  linear programming bound was computed with polynomials of degree~30.  Packings
  with the best radius in rows marked with a~$\star$ are known to be
  optimal~\cite{Cohn2024}.}%
  \label{tab:caps}
\end{table}

Schütte and van der Waerden~\cite{SchutteW1951} constructed optimal spherical
cap packings with~5--9 caps.  Only a few other optimal packings are
known~\cite{Cohn2024}.  To test the performance of the manifold optimization
method of~\S\ref{sec:manifold-opt}, we used it to find spherical cap packings
with~5--24 caps.  The configurations found numerically closely approximate the
best known configurations, some of which are optimal. The results are summarized
in Table~\ref{tab:caps}, where the theta prime upper bound (that is, the linear
programming bound of Delsarte, Goethals, and Seidel~\cite{DelsarteGS1977}) and
the volume bound for each configuration are also given.  These results indicate
that the manifold optimization approach can find near-optimal packings and also
what can be expected of the theta prime upper bound.

Table~\ref{tab:bounds} summarizes our results for packings of spherical
polygons.  Below is a description of each column.

\begin{description}
  \item[$N$] number of sides of the polygon being packed.

  \item[Radius] the angular radius of the polygon, in radians.

  \item[Lower bound] best known lower bound, coming from a construction.

  \item[Volume bound] the volume bound.

  \item[3-point bound] The 3-point~\cite{BachocV2008} bound for a packing of
    spherical caps with radius~$\tau$, where~$\tau$ is the inradius of the
    polygon (see~\S\ref{sec:spherical-codes}).  This bound is computed with
    polynomials of total degree~$28$.  If there is a packing of spherical caps
    with radius at least~$\tau$ and cardinality greater than the lower bound,
    then the cardinality of such a packing is given in parenthesis.  In this
    case, spherical code bounds can never prove maximality of the configuration.

  \item[$d$] Degree parameter~$d$ for the theta prime number of~\S\ref{sec:sos}.

  \item[$\vartheta'$] The value of~$\vartheta'(G(K))$ computed using the
    semidefinite programming formulation of~\S\ref{sec:sos}.
\end{description}

A $\dagger$ indicates a configuration that is proven maximal by the theta prime
bound, but also by another bound (the volume bound or the 3-point bound),
whereas a~$\star$ indicates a configuration that is proven maximal by the theta
prime bound alone.  A~$-$ indicates that the theta prime bound is better than
the floor of the volume bound and 3-point bound, though it does not prove
maximality.
\bigbreak

\begingroup
  \footnotesize
  \def\ct#1{\multicolumn{1}{c}{#1}}
  \setlength\LTleft{-1000pt plus 1fill}
  \setlength\LTright{-1000pt plus 1fill}

  \begin{longtable}{cccrrrrl}
    \tabularnewline
    \caption{Lower and upper bounds for packings of regular spherical
    polygons.  The radius is truncated to~6 decimal places; other numbers are
    truncated to~8 decimal places.}
    \endfoot
          &                                     &                       &\multicolumn{4}{c}{\sl Upper bounds}                                                    &\\
    $N$   &   {\sl Radius}                      &   {\sl Lower bound}   & \ct{\sl Volume bound}  &  \ct{\sl 3-point bound}   & \ct{$d$}&  \ct{$\vartheta'$}      &   {\sl Remarks}\\
    \hline\noalign{\rule{0pt}{1pt}}
    \endhead
    \rowcolor{gray!20}
    3\rlap{${}^\dagger$}&0.955316&8&8.00000000&(9) 9.00000000&14&8.00148720&conjectured\\
    3&0.846584&9&10.84079445&(12) 12.72238983&14&10.43859589&conjectured\\
    \rowcolor{gray!20}
    3\rlap{${}^\star$}&0.828930&10&11.41373633&(12) 13.49278286&14&10.97882053&conjectured\\
    3&0.790798&11&12.78772029&(15) 15.65588685&14&12.31168623&conjectured\\
    \rowcolor{gray!20}
    3\rlap{${}^-$}&0.760924&12&14.01415753&(17) 17.34929797&14&13.55069849&conjectured\\*
    \rowcolor{gray!20}
    &&&&&20&13.53801952&\\
    3\rlap{${}^-$}&0.738005&13&15.05923710&(18) 18.94361211&14&14.62806719&conjectured\\
    \rowcolor{gray!20}
    3&0.719838&14&15.96041040&(20) 20.34188158&14&15.57438551&\\
    3\rlap{${}^-$}&0.699345&15&17.06344004&(20) 21.78736915&14&16.77117697&\\
    \rowcolor{gray!20}
    3&0.685159&16&17.88615258&(22) 23.08083263&14&17.67500325&\\*
    \rowcolor{gray!20}
    &&&&&20&17.62851173&\\
    3&0.669803&17&18.83686063&(24) 24.57551675&14&18.71223464&\\*
    &&&&&20&18.67127997&\\
    \rowcolor{gray!20}
    3&0.658112&18&19.60600155&(24) 25.66413630&14&19.57509688&\\*
    \rowcolor{gray!20}
    &&&&&20&19.55107105&\\
    3\rlap{${}^\dagger$}&0.652634&19&19.98080594&(24) 26.21183173&14&19.99058786&\\*
    &&&&&20&19.98296276&\\
    \rowcolor{gray!20}
    3&0.615479&20&22.79466514&(28) 30.43073692&14&22.59030673&tetrahedron kissing\\*
    \rowcolor{gray!20}
    &&&&&20&22.51176943&\\
    3\rlap{${}^\dagger$}&0.651805&20&20.03839317&(24) 26.29552502&14&20.04764259&\\*
    &&&&&20&20.03822089&\\
    \rowcolor{gray!20}
    3&0.604417&21&23.73485658&(30) 31.77968284&14&23.48417323&\\*
    \rowcolor{gray!20}
    &&&&&20&23.40372464&\\
    3\rlap{${}^-$}&0.587346&22&25.29181773&(32) 34.02763309&14&24.99197449&\\*
    &&&&&20&24.90075396&\\
    \rowcolor{gray!20}
    3&0.570870&23&26.92948343&(34) 36.66255105&14&26.59150848&\\*
    \rowcolor{gray!20}
    &&&&&20&26.48770382&\\
    3\rlap{${}^-$}&0.558125&24&28.29737925&(36) 38.71347302&14&27.93923717&\\*
    &&&&&20&27.82052189&\\
    \hline\noalign{\vskip1pt}\hline
    \rowcolor{gray!20}
    4\rlap{${}^\dagger$}\rule{0pt}{2.6ex}&1.183199&3&3.70443713&3.00000000&10&3.00000001&\\
    4\rlap{${}^\dagger$}&1.063440&4&4.71541086&4.07790842&10&4.10116100&\\
    \rowcolor{gray!20}
    4&0.955316&5&6.00000117&(6) 6.00000000&10&6.00000044&\\*
    \rowcolor{gray!20}
    &&&&&18&6.00000026&\\
    4\rlap{${}^\dagger$}&0.955316&6&6.00000000&6.00000000&8&6.00000463&cube tiling\\
    \rowcolor{gray!20}
    4\rlap{${}^-$}&0.785398&7&9.24441273&(8) 9.00000000&10&8.46487808&octahedron kissing\\*
    \rowcolor{gray!20}
    &&&&&18&8.43452493&\\
    4\rlap{${}^\star$}&0.812554&7&8.58304028&(8) 8.13119410&10&7.84845101&\\
    \rowcolor{gray!20}
    4\rlap{${}^\star$}&0.777160&8&9.45911352&(9) 9.10977817&10&8.67678741&\\
    4\rlap{${}^\star$}&0.737706&9&10.59062399&(10) 10.27965041&10&9.72908492&\\
    \rowcolor{gray!20}
    4\rlap{${}^\star$}&0.701465&10&11.80498479&(12) 12.08022888&10&10.87300195&\\
    4\rlap{${}^-$}&0.665589&11&13.20988402&(12) 13.08564377&10&12.36460904&\\*
    &&&&&18&12.24240918&\\
    \rowcolor{gray!20}
    4\rlap{${}^\star$}&0.659825&12&13.45739970&13.35693928&10&12.61827032&\\
    4\rlap{${}^-$}&0.611102&13&15.83957294&(15) 16.21135247&10&14.96350191&\\*
    &&&&&18&14.70525606&\\
    \rowcolor{gray!20}
    4\rlap{${}^-$}&0.598860&14&16.53186711&(16) 16.95237174&10&15.70271462&\\*
    \rowcolor{gray!20}
    &&&&&18&15.36791722&\\
    4\rlap{${}^-$}&0.585592&15&17.33205714&(17) 17.82667450&10&16.60940922&\\*
    &&&&&18&16.16157960&\\
    \rowcolor{gray!20}
    4&0.564943&16&18.69188151&(18) 19.54161047&10&18.18025701&\\
    4\rlap{${}^-$}&0.549177&17&19.83538800&(20) 20.76935334&10&19.48304253&\\*
    &&&&&18&18.59577345&\\
    \rowcolor{gray!20}
    4\rlap{${}^-$}&0.537891&18&20.71669278&(20) 21.67448098&10&20.53747812&\\*
    \rowcolor{gray!20}
    &&&&&18&19.43932673&\\
    4\rlap{${}^-$}&0.516950&19&22.50799105&(22) 23.89372360&10&22.82038760&\\*
    &&&&&18&21.21727058&\\
    \rowcolor{gray!20}
    4\rlap{${}^-$}&0.510062&20&23.14608098&(24) 24.61734992&10&23.66899224&\\*
    \rowcolor{gray!20}
    &&&&&18&21.84775700&\\
    4\rlap{${}^-$}&0.495188&21&24.61617987&(24) 26.24964661&10&25.64625742&\\*
    &&&&&18&23.29803247&\\
    \rowcolor{gray!20}
    4\rlap{${}^-$}&0.485876&22&25.60629269&(26) 27.32759700&10&27.00764135&\\*
    \rowcolor{gray!20}
    &&&&&18&24.26292605&\\
    4\rlap{${}^-$}&0.475163&23&26.81816363&(27) 28.70062328&10&28.74508182&\\*
    &&&&&18&25.44302605&\\
    \rowcolor{gray!20}
    4\rlap{${}^-$}&0.465337&24&28.00434523&(28) 30.17773550&10&30.52343961&\\*
    \rowcolor{gray!20}
    &&&&&18&26.62321441&\\
    \hline\noalign{\vskip1pt}\hline
    5\rlap{${}^\dagger$}\rule{0pt}{2.6ex}&1.107148&3&4.00000000&3.23606797&10&3.23606979&\\*
    &&&&&14&3.23606992&\\
    \rowcolor{gray!20}
    5\rlap{${}^\dagger$}&1.017221&4&4.76984107&4.04611156&10&4.06501130&icosahedron kissing\\
    5\rlap{${}^\dagger$}&1.027058&4&4.67521434&4.02586260&10&4.00665492&\\*
    &&&&&14&4.00425871&\\
    \rowcolor{gray!20}
    5\rlap{${}^\star$}&0.865924&5&6.67214088&(6) 6.01513743&10&5.95011752&\\*
    \rowcolor{gray!20}
    &&&&&14&5.94289949&\\
    5\rlap{${}^\dagger$}&0.850308&6&6.92983044&6.05170751&10&6.23338399&\\*
    &&&&&14&6.22491869&\\
    \rowcolor{gray!20}
    5\rlap{${}^\star$}&0.756498&7&8.83521938&(8) 8.00493858&10&8.03932801&\\*
    \rowcolor{gray!20}
    &&&&&14&7.91739923&\\
    5\rlap{${}^\dagger$}&0.731849&8&9.46296954&8.44605070&10&8.76975807&\\*
    &&&&&14&8.63164917&\\
    \rowcolor{gray!20}
    5\rlap{${}^\dagger$}&0.692196&9&10.61854972&9.56183171&10&10.10476653&\\*
    \rowcolor{gray!20}
    &&&&&14&9.99379488&\\
    5&0.658224&10&11.78061879&11.03162960&10&11.62203944&\\*
    &&&&&14&11.57505633&\\*
    &&&&&20&11.53768330&\\
    \rowcolor{gray!20}
    5&0.652358&11&12.00000004&(12) 12.00000000&10&12.00002213&\\*
    \rowcolor{gray!20}
    &&&&&14&12.00000393&\\
    5\rlap{${}^\dagger$}&0.652358&12&12.00000510&12.00000001&10&12.00002878&\\
    \rowcolor{gray!20}
    5&0.565329&13&16.10441278&(14) 15.26026568&10&16.02328802&\\*
    \rowcolor{gray!20}
    &&&&&14&15.32016393&\\
    5&0.554476&14&16.75654644&(15) 15.89359509&10&16.78516545&\\*
    &&&&&14&15.95264213&\\
    \rowcolor{gray!20}
    5&0.539673&15&17.71049491&(16) 16.83442133&10&17.93479927&\\*
    \rowcolor{gray!20}
    &&&&&14&16.89037273&\\
    5&0.524729&16&18.75684983&(17) 17.87168328&10&19.26071663&\\*
    &&&&&14&17.89517658&\\
    \rowcolor{gray!20}
    5&0.512308&17&19.69727207&(18) 18.91986202&10&20.52742197&\\*
    \rowcolor{gray!20}
    &&&&&14&18.82818266&\\
    5\rlap{${}^-$}&0.498146&18&20.85668953&(19) 20.21849698&10&22.19285776&\\*
    &&&&&14&19.96762912&\\*
    &&&&&20&19.62178909&\\
    \rowcolor{gray!20}
    5&0.480722&19&22.42647447&(20) 21.66650360&10&24.64563342&\\*
    \rowcolor{gray!20}
    &&&&&14&21.49920555&\\*
    \rowcolor{gray!20}
    &&&&&20&21.05564494&\\
    5\rlap{${}^-$}&0.476634&20&22.81997390&(21) 22.05813478&10&25.30113061&\\*
    &&&&&14&21.88236952&\\*
    &&&&&20&21.39579388&\\
    \rowcolor{gray!20}
    5\rlap{${}^-$}&0.459979&21&24.53327575&(24) 24.03727825&10&28.35052843&\\*
    \rowcolor{gray!20}
    &&&&&14&23.60037910&\\*
    \rowcolor{gray!20}
    &&&&&20&22.91769941&\\
    5\rlap{${}^-$}&0.451462&22&25.48362122&(24) 24.98748702&10&30.21979777&\\*
    &&&&&14&24.58375506&\\*
    &&&&&20&23.81047251&\\
    \rowcolor{gray!20}
    5\rlap{${}^-$}&0.436889&23&27.24080150&(25) 26.77138644&10&33.93749530&\\*
    \rowcolor{gray!20}
    &&&&&14&26.51402878&\\*
    \rowcolor{gray!20}
    &&&&&20&25.41736863&\\
    5\rlap{${}^-$}&0.435221&24&27.45326251&(25) 26.98022302&10&34.40324778&\\*
    &&&&&14&26.75588236&\\*
    &&&&&20&25.61687007&
\label{tab:bounds}  
\end{longtable}
\endgroup


All lower bounds in Table~\ref{tab:bounds} have been computed using the manifold
optimization method of~\S\ref{sec:manifold-opt}, except for the tetrahedron
kissing number and the cube tiling configurations.  The optimization was
implemented in Julia using the packages \texttt{Manopt} and
\texttt{Manifolds}~\cite{AxenBBR2023, Bergmann2022}.  The results have been
checked and are rigorous unless otherwise noted; the radius given in the table
has been truncated to~6 decimal places. The lower bounds for pentagons found
through manifold optimization approximate closely those by Tarnai and
Gáspár~\cite{TarnaiG2001}, except for the packing of~9 pentagons. For this case,
the manifold optimization method gives a significantly larger radius for the
polygons.

The theta prime upper bounds have been computed by a Julia program using the
\texttt{ClusteredLowRankSolver} package~\cite{LeijenhorstL2024}.  Computations
for~$N = 3$ and $d =14$ require a few hours; for~$d = 20$ the computation
requires over~6 days and~128GB of memory.  Computations have been carried out in
high-precision floating-point arithmetic; the bounds obtained can be rigorously
verified using standard techniques~\cite{LaatOV2014}.

Some bounds for triangles are marked ``conjectured'' in the table.  This means
that the triangle being packed is not admissible according to
Theorem~\ref{thm:admissible-andreas}.  The inner angles of their polars are at
least~$\pi/2$, however, and so if Conjecture~\ref{conj:admissible} is true then
the theta prime upper bounds for these entries of the table are valid.

A disappointing finding is that the theta prime number for the tetrahedron
kissing number problem is barely better than the volume bound.  The theta prime
number is sharp for the packing of faces of the cube, which tile the sphere.  In
this case, as in the case of every tiling of the sphere, the volume bound is
sharp, and hence the theta prime number is also sharp.  It is at first sight
noteworthy however that a sharp solution exists for~$d = 8$.  This fact is less
surprising once one notes that the linear programming bound for the
corresponding spherical-cap-packing problem is sharp, what also explains why the
3-point bound is sharp in this case.

In many other cases, the theta prime number is enough to prove maximality of the
corresponding configuration.  Sometimes the 3-point bound also proves
maximality, but in some cases no bound from spherical codes can prove maximality
of the configuration, because for the corresponding inradius there is a
spherical code with larger cardinality than that of the configuration (see the
table of spherical codes maintained by Cohn~\cite{Cohn2024}).

Determining the kissing number of the tetrahedron was our original motivation.
We could consider the kissing number of the other Platonic solids: how many
cubes, octahedra, dodecahedra, or icosahedra can share a vertex in common?  This
question reduces to finding the packing number of the cones corresponding to the
vertex figures of each polytope.  For the cube, the answer is clear. For the
dodecahedron, the corresponding cone gives a spherical triangle that is not
admissible because its polar has inner angle smaller than~$\pi/2$, and so our
methods fail.  For the octahedron and icosahedron we get lower and upper bounds,
indicated in Table~\ref{tab:bounds}.  These bounds show that the kissing number
of the octahedron is either~7 or~8 and that the kissing number of the
icosahedron is~4.

Finally, our results suggest that the obstacle for proving optimality or
near-optimality of configurations is the quality of the theta prime upper bound,
which is even sometimes very close to the volume bound.  This leads to the most
obvious future direction of research: implementing higher-order bounds for
packings of regular spherical polygons.  Natural candidates here are the 3-point
bound, which has been described for general topological packing graphs by de
Laat, Machado, Oliveira, and Vallentin~\cite{LaatMOV2022}, or the Lasserre
hierarchy, which was extended to topological packing graphs by de Laat and
Vallentin~\cite{LaatV2015}.  The challenge here is not only to develop the
required harmonic analysis theory, but also to simplify the resulting problem to
make it tractable.

All programs used in the computation of the bounds as well as a complete
database with all the configurations and bounds can be found on Harvard
Dataverse~\cite{OliveiraSV2026}.


\section*{Acknowledgments}

We would like to thank Fabrício Caluza Machado, who did preliminary work on the
theta prime number for~$\sorto(3)$ upon which we build, Karla Leipold for
help with manifold optimization, and Martin Olschewski for facilitating access
to computational resources at the University of Cologne.

Upper bound computations were carried out on two clusters at the Delft Institute
of Applied Mathematics, managed by Joffrey Wallaart, and at the TU Delft
supercomputer DelftBlue provided by the Delft High Performance Computing
Centre\footnote{\url{https://www.tudelft.nl/dhpc}}.


{

}


\frenchspacing

\begin{thebibliography}{44}
\bibitem{Addabbo2015}
C.~Addabbo, {\it Il Libellus de impletione loci di Francesco Maurolico
e la tassellazione dello spazio}, PhD Thesis, Università di Pisa,
2015, 299pp.

\bibitem{AndrewsAR1999}
G.E.~Andrews, R.~Askey, and R.~Roy, {\it Special functions},
Encyclopedia of Mathematics and its Applications~71, Cambridge
University Press, Cambridge, 1999.

\bibitem{AxenBBR2023}
S.D.~Axen, M.~Baran, R.~Bergmann, and K.~Rzecki, {\tt Manifolds.jl}:
an extensible Julia framework for data analysis on manifolds, {\it ACM
Trans. Math. Software\/}~49 (2023) Art. 33, 23.

\bibitem{BachocGSV2012}
C.~Bachoc, D.~Gijswijt, A.~Schrijver, and F.~Vallentin, Invariant
semidefinite programs, in: {\it Handbook on semidefinite, conic and
polynomial optimization}, Internat. Ser. Oper. Res. Management
Sci.~166, Springer, New York, 2012, pp.~219--269.

\bibitem{BachocNOV2009}
C.~Bachoc, G.~Nebe, F.M.~de Oliveira Filho, and F.~Vallentin, Lower
bounds for measurable chromatic numbers, {\it Geom. Funct. Anal.\/}~19
(2009) 645--661.

\bibitem{BachocV2008}
C.~Bachoc and F.~Vallentin, New upper bounds for kissing numbers from
semidefinite programming, {\it J. Amer. Math. Soc.\/}~21 (2008)
909--924.

\bibitem{Bergmann2022}
R.~Bergmann, Manopt.jl: optimization on manifolds in Julia, {\it
Journal of Open Source Software\/}~7 (2022) 3866.

\bibitem{BirginM2014}
E.G.~Birgin and J.M.~Martínez, {\it Practical augmented Lagrangian
methods for constrained optimization}, Fundamentals of Algorithms~10,
Society for Industrial and Applied Mathematics (SIAM), Philadelphia,
PA, 2014.

\bibitem{BirginMMR2006}
E.G.~Birgin, J.M.~Martínez, W.F.~Mascarenhas, and D.P.~Ronconi, Method
of sentinels for packing items within arbitrary convex regions, {\it
Journal of the Operational Research Society\/}~57 (2006) 735--746.

\bibitem{Bochner1941}
S.~Bochner, Hilbert distances and positive definite functions, {\it
Ann. of Math. (2)\/}~42 (1941) 647--656.

\bibitem{Boumal2023}
N.~Boumal, {\it An introduction to optimization on smooth manifolds},
Cambridge University Press, Cambridge, 2023.

\bibitem{BrockerD1985}
T.~Bröcker and T.~tom Dieck, {\it Representations of compact Lie
groups}, Graduate Texts in Mathematics~98, Springer-Verlag, New York,
1985.

\bibitem{Casselman2004}
B.~Casselman, The difficulties of kissing in three dimensions, {\it
Notices Amer. Math. Soc.\/}~51 (2004) 884--885.

\bibitem{ChirikjianK2001}
G.S.~Chirikjian and A.B.~Kyatkin, {\it Engineering applications of
noncommutative harmonic analysis\/} (With emphasis on rotation and
motion groups), CRC Press, Boca Raton, FL, 2001.

\bibitem{Cohn2024}
H.~Cohn, Table of spherical codes, {\tt
https://hdl.handle.net/1721.1/153543}, DSpace@MIT, 2024.

\bibitem{CohnE2003}
H.~Cohn and N.~Elkies, New upper bounds on sphere packings. I, {\it
Ann. of Math. (2)\/}~157 (2003) 689--714.

\bibitem{CohnLL2024}
H.~Cohn, D.~de Laat, and N.M.~Leijenhorst, Optimality of spherical
codes via exact semidefinite programming bounds, arXiv:2403.16874,
2024, 32pp.

\bibitem{DelsarteGS1977}
P.~Delsarte, J.M.~Goethals, and J.J.~Seidel, Spherical codes and
designs, {\it Geometriae Dedicata\/}~6 (1977) 363--388.

\bibitem{DostertGOV2017}
M.~Dostert, C.~Guzmán, F.M.~de Oliveira Filho, and F.~Vallentin, New
upper bounds for the density of translative packings of
three-dimensional convex bodies with tetrahedral symmetry, {\it
Discrete Comput. Geom.\/}~58 (2017) 449--481.

\bibitem{Dumitrescu2017}
B.~Dumitrescu, {\it Positive trigonometric polynomials and signal
processing applications}, Signals and Communication Technology,
Springer, Cham, 2017.

\bibitem{Folland1995}
G.B.~Folland, {\it A Course in Abstract Harmonic Analysis}, Studies in
Advanced Mathematics, CRC Press, Boca Raton, FL, 1995.

\bibitem{GatermannP2004}
K.~Gatermann and P.A.~Parrilo, Symmetry groups, semidefinite programs,
and sums of squares, {\it J. Pure Appl. Algebra\/}~192 (2004) 95--128.

\bibitem{KabatianskyL1978}
G.A.~Kabatiansky and V.I.~Levenshtein, Bounds for packings on the
sphere and in space, {\it Problemy Peredači Informacii\/}~14 (1978)
3--25.

\bibitem{LaatMOV2022}
D.~de Laat, F.C.~Machado, F.M.~de Oliveira Filho, and F.~Vallentin,
$k$-point semidefinite programming bounds for equiangular lines, {\it
Math. Program.\/}~194 (2022) 533--567.

\bibitem{LaatOV2014}
D.~de Laat, F.M.~de Oliveira Filho, and F.~Vallentin, Upper bounds for
packings of spheres of several radii, {\it Forum Math. Sigma\/}~2
(2014) Paper No. e23, 42.

\bibitem{LaatV2015}
D.~de Laat and F.~Vallentin, A semidefinite programming hierarchy for
packing problems in discrete geometry, {\it Math. Program.\/}~151
(2015) 529--553.

\bibitem{LagariasZ2012}
J.C.~Lagarias and C.~Zong, Mysteries in packing regular tetrahedra,
{\it Notices Amer. Math. Soc.\/}~59 (2012) 1540--1549.

\bibitem{LeijenhorstL2024}
N.M.~Leijenhorst and D.~de Laat, Solving clustered low-rank
semidefinite programs arising from polynomial optimization, {\it Math.
Program. Comput.\/}~16 (2024) 503--534.

\bibitem{LiuB2020}
C.~Liu and N.~Boumal, Simple algorithms for optimization on Riemannian
manifolds with constraints, {\it Appl. Math. Optim.\/}~82 (2020)
949--981.

\bibitem{Lovasz1979}
L.~Lovász, On the Shannon capacity of a graph, {\it IEEE Trans.
Inform. Theory\/}~25 (1979) 1--7.

\bibitem{MascarenhasB2010}
W.F.~Mascarenhas and E.G.~Birgin, Using sentinels to detect
intersections of convex and nonconvex polygons, {\it Comput. Appl.
Math.\/}~29 (2010) 247--267.

\bibitem{McElieceRR1978}
R.J.~McEliece, E.R.~Rodemich, and H.C.~Rumsey, The Lovász bound and
some generalizations, {\it J. Combin. Inform. System Sci.\/}~3 (1978)
134--152.

\bibitem{MusinT2015}
O.R.~Musin and A.S.~Tarasov, The Tammes problem for $N=14$, {\it Exp.
Math.\/}~24 (2015) 460--468.

\bibitem{MusinT2012}
O.R.~Musin and A.S.~Tarasov, The strong thirteen spheres problem, {\it
Discrete Comput. Geom.\/}~48 (2012) 128--141.

\bibitem{OliveiraSV2026}
F.M.~de Oliveira Filho, A.~Spomer, and F.~Vallentin, Data and code
for: Bounding the density of spherical polygon packings, {\tt
https://doi.org/10.7910/DVN/PNH45K}, Harvard Dataverse, 2026, V1.

\bibitem{OliveiraV2018}
F.M.~de Oliveira Filho and F.~Vallentin, Computing upper bounds for
the packing density of congruent copies of a convex body, in: {\it New
trends in intuitive geometry}, Bolyai Soc. Math. Stud.~27, János
Bolyai Math. Soc., Budapest, 2018, pp.~155--188.

\bibitem{OliveiraV2010}
F.M.~de Oliveira Filho and F.~Vallentin, Fourier analysis, linear
programming, and densities of distance avoiding sets in $\Bbb R^n$,
{\it J. Eur. Math. Soc. (JEMS)\/}~12 (2010) 1417--1428.

\bibitem{Schrijver1979}
A.~Schrijver, A comparison of the Delsarte and Lovász bounds, {\it
IEEE Trans. Inform. Theory\/}~25 (1979) 425--429.

\bibitem{Schrijver2003B}
A.~Schrijver, {\it Combinatorial optimization. Polyhedra and
efficiency. Vol. B\/} (Matroids, trees, stable sets, Chapters 39--69),
Algorithms and Combinatorics~24, Springer-Verlag, Berlin, 2003.

\bibitem{SchutteW1951}
K.~Schütte and B.L.~van der Waerden, Auf welcher Kugel haben $5$, $6$,
$7$, $8$ oder $9$ Punkte mit Mindestabstand Eins Platz?, {\it Math.
Ann.\/}~123 (1951) 96--124.

\bibitem{SchutteW1953}
K.~Schütte and B.L.~van der Waerden, Das Problem der dreizehn Kugeln,
{\it Math. Ann.\/}~125 (1953) 325--334.

\bibitem{Serre1977}
J.~Serre, {\it Linear representations of finite groups\/} (Translated
from the second French edition by Leonard L. Scott), Graduate Texts in
Mathematics, Vol. 42, Springer-Verlag, New York-Heidelberg, 1977.

\bibitem{Spomer2026}
A.~Spomer, {\it Packing Regular Spherical Polygons\/} (to appear), PhD
Thesis, University of Cologne, 2026.

\bibitem{TarnaiG2001}
T.~Tarnai and Zs.~Gáspár, Packing of equal regular pentagons on a
sphere, {\it R. Soc. Lond. Proc. Ser. A Math. Phys. Eng. Sci.\/}~457
(2001) 1043--1058.

\end{thebibliography}
\end{document}